\documentclass[reqno]{amsart}
\usepackage{amsmath}
\usepackage{amsfonts}
\usepackage{amssymb,enumerate,comment,mathdots, stmaryrd}
\usepackage{amsthm}
\usepackage[all]{xy}
\usepackage{rotating}
\usepackage{hyperref}

\usepackage{mathrsfs}

\usepackage{graphicx}
\usepackage{tcolorbox}
\usepackage{color}

\theoremstyle{plain}
\newtheorem{lem}{Lemma}[section]
\newtheorem{cor}[lem]{Corollary}
\newtheorem{prop}[lem]{Proposition}
\newtheorem{thm}[lem]{Theorem}

\theoremstyle{definition}
\newtheorem{defn}[lem]{Definition}
\newtheorem{ex}[lem]{Example}

\newtheorem{notn}[lem]{Notation}
\newtheorem{fact}[lem]{Fact}

\newtheorem{step}{Step}

\theoremstyle{remark}
\newtheorem{case}{Case}

\newcommand{\ext}{\operatorname{Ext}}

\newcommand{\Hom}{\operatorname{Hom}}

\newcommand{\s}{\mathfrak{S}}
\newcommand{\tor}{\operatorname{Tor}}

\newcommand{\Cl}{\operatorname{Cl}}

\newcommand{\Ker}{\operatorname{Ker}}

\newcommand{\ideal}[1]{\mathfrak{#1}}
\newcommand{\m}{\ideal{m}}

\newcommand{\p}{\ideal{p}}
\newcommand{\q}{\ideal{q}}

\newcommand{\fa}{\ideal{a}}
\newcommand{\fb}{\ideal{b}}

\newcommand{\sfk}{\mathsf k}

\newcommand{\ti}{\tilde}

\newcommand{\bbz}{\mathbb{Z}}
\newcommand{\bbn}{\mathbb{N}}

\newcommand{\xra}{\xrightarrow}
\newcommand{\xla}{\xleftarrow}

\newcommand{\vf}{\varphi}
\newcommand{\ve}{\varepsilon}

\newcommand{\te}{\theta}

\newcommand{\gd}{\delta}
\newcommand{\gl}{\lambda}

\newcommand{\frakp}{\mathfrak{p}}
\newcommand{\frakq}{\mathfrak{q}}

\newcommand{\hhom}{\operatorname{Hom}}

\renewcommand{\geq}{\geqslant}
\renewcommand{\leq}{\leqslant}

\renewcommand{\ker}{\Ker}

\newcommand{\Ext}[4][R]{\operatorname{Ext}_{#1}^{#2}(#3,#4)}

\newcommand{\Otimes}[3][R]{#2\otimes_{#1}#3}
\renewcommand{\Hom}[3][R]{\operatorname{Hom}_{#1}(#2,#3)}	
\newcommand{\Tor}[4][R]{\operatorname{Tor}^{#1}_{#2}(#3,#4)}

\newcommand{\ssm}{\smallsetminus}

\numberwithin{equation}{lem}

\begin{document}

\bibliographystyle{amsplain}

\author{Sean K. Sather-Wagstaff}

\address{Department of Mathematical Sciences,
Clemson University,
O-110 Martin Hall, Box 340975, Clemson, S.C. 29634
USA}

\email{ssather@clemson.edu}

\urladdr{https://ssather.people.clemson.edu/}

\author{Tony Se}

\address{Department of Mathematics,
University of Mississippi,
Hume Hall 335, P.O. Box 1848, University, MS 38677
USA}

\email{ttse@olemiss.edu}

\urladdr{http://math.olemiss.edu/tony-se/}

\author{Sandra Spiroff}

\address{Department of Mathematics,
University of Mississippi,
Hume Hall 335, P.O. Box 1848, University, MS 38677
USA}

\email{spiroff@olemiss.edu}

\urladdr{http://math.olemiss.edu/sandra-spiroff/}

\thanks{
Sandra Spiroff was supported in part by Simons Foundation Collaboration Grant 245926.}

\title{Semidualizing modules of $2 \times 2$ Ladder Determinantal Rings}

\date{\today}

\dedicatory{}


\begin{abstract} We continue our study of ladder determinantal rings over a field $\sfk$ from the perspective of semidualizing modules.  In particular, given a ladder of variables $Y$, we show that the associated ladder determinantal ring $\sfk[Y]/I_2(Y)$ admits exactly $2^n$ non-isomorphic semidualizing modules where $n$ is determined from the combinatorics of the ladder $Y$: the number $n$ is essentially the number of non-Gorenstein factors in a certain decomposition of  $Y$. From this, for each $n$, we show explicitly how to find ladders $Y$ such that $\sfk[Y]/I_2(Y)$ admits exactly $2^n$ non-isomorphic semidualizing modules. This is in contrast to our previous work, which demonstrates that large classes of ladders have exactly 2 non-isomorphic semidualizing modules.
\end{abstract}

\maketitle

\section{Introduction} \label{sec1}

Let $R$ be a commutative noetherian ring and let $\sfk$ be a field. We are interested in the question of how many non-isomorphic semidualizing modules the ring $R$ has, where a finitely generated $R$-module $C$ is \textbf{semidualizing} if $\Hom CC\cong R$ and $\Ext iCC=0$ for all $i\geq 1$. Examples include the free $R$-module of rank 1 and, if $R$ is local and Cohen-Macaulay, a canonical or dualizing module. This general class of modules was introduced by Foxby~\cite{foxby:gmarm} in part to understand Auslander and Bridger's G-dimension~\cite{auslander:smt} and non-finitely generated versions due to Enochs and Jenda~\cite{enochs:gipm}. 
The number of non-isomorphic semidualizing $R$-modules is finite when $R$ is local or a standard graded normal domain by Nasseh and Sather-Wagstaff~\cite{nasseh:gart,SSW}. In each of these cases, the number of these modules measures how far $R$ is from being Gorenstein.

The question \cite[Question~4.13]{SSW} that we address in this paper relates to the cardinality of the set of semidualizing modules $\frak S_0(R)$ of a Noetherian local ring $R$: Must it always be a power of 2?  We provide more evidence towards an affirmative answer in the case of ladder determinantal rings. These rings generalize the classical determinantal rings in a way that is useful, e.g., for studying Young tableaux~\cite{MR926272}.

Referring back to the detailed background section in \cite{SWSeSpP1}, we briefly recall that a {\bf ladder} is a subset $Y$ of an $m \times n$ matrix $X=(X_{ij})$ of indeterminates satisfying the property that if $X_{ij},X_{pq}\in Y$ satisfy $i\leq p$ and $j\leq q$, then $X_{iq},X_{pj}\in Y$.  Then $R_2(Y) = \sfk[Y]/I_2(Y)$ is the associated {\bf ladder determinantal ring} of 2-minors, as $I_2(Y)$ is the ideal generated by the $2 \times 2$ minors lying entirely in $Y$.  To avoid trivialities, we assume that $X$ is the smallest matrix containing $Y$ and that every variable of $Y$ is part of a $2 \times 2$ minor.  To wit, we consider only {\bf 2-connected} ladders; i.e., those ladders $Y$ satisfying the property that there do not exist two subladders $\varnothing \neq Z_1, Z_2 \subseteq Y$ such that $Z_1 \cap Z_2 = \varnothing$, $Z_1 \cup Z_2 = Y$, and every $2$-minor of $Y$ is contained in $Z_1$ or $Z_2$.  

We provide a construction to produce ladder determinantal rings with exactly $2^n$ semidualizing modules for any $n \in \mathbb N$, and we show that $|\frak S_0(R_2(Y))|$ is always, in fact, a power of 2. To describe our results explicitly, we use the corners of a ladder; see Definition~\ref{defn181011a} and the sample ladders below. If a lower and upper inside corner of a ladder $Y$ coincide, then we say that $Y$ has a coincidental inside corner.

\begin{table}[h]
  \begin{center}
    \begin{tabular}{c c c c c c c c c c c c c c c c c c c c c c c c c c}
     & \!\!\!\!\!$X_{12}$ & \!\!\!\!\!$X_{13}$ & & & & & &  \!\!\!\!\!$X_{12}$ & \!\!\!\!\!$X_{13}$ & \!\!\!\!\!$X_{14}$ & \!\!\!\!\!$X_{15}$ & & & & & & & & & \!\!\!\!\!$X_{12}$ & \!\!\!\!\!$X_{13}$ \\
      \!\!\!\!\!$X_{21}$ & \!\!\!\!\!$X_{22}$ & \!\!\!\!\!$X_{23}$ & & & & & & \!\!\!\!\!$X_{22}$ & \!\!\!\!\!$X_{23}$ & \!\!\!\!\!$X_{24}$ & \!\!\!\!\!$X_{25}$ & & & & & & & & & \!\!\!\!\!$X_{22}$ & \!\!\!\!\!$X_{23}$ \\
      \!\!\!\!\!$X_{31}$ & \!\!\!\!\!$X_{32}$ & \!\!\!\!\!$X_{33}$ & & & & & & \!\!\!\!\!$X_{32}$ & \!\!\!\!\!$X_{33}$ & \!\!\!\!\!$X_{34}$ & & & & & & & & & \!\!\!\!\!$X_{31}$ & \!\!\!\!\!$X_{32}$ & \!\!\!\!\!$X_{33}$ \\
       \!\!\!\!\!$X_{41}$ & \!\!\!\!\!$X_{42}$ & & & & & & \!\!\!\!\!$X_{41}$ & \!\!\!\!\!$X_{42}$ & \!\!\!\!\!$X_{43}$ & & & & & & & & & & \!\!\!\!\!$X_{41}$ & \!\!\!\!\!$X_{42}$ \\
        \!\!\!\!\!$X_{51}$ & \!\!\!\!\!$X_{52}$ & & & & & & \!\!\!\!\!$X_{51}$ & \!\!\!\!\!$X_{52}$ & \!\!\!\!\!$X_{53}$ & & & & & & & & & & \!\!\!\!\!$X_{51}$ & \!\!\!\!\!$X_{52}$ \\
        & & & &\\
        & $L_1$ & & & & & & & & $L_2$  & & & & & & & & & & & $L_3$ \\
    \end{tabular}
  \end{center}
\end{table}
\begin{tabular}{c c c c c c c c c c c c c c c}
& & {\it ladders with only non-coincidental} & & & & & & {\it a ladder with a}  \\
& & {\it inside corners} & & & & & & {\it coincidental inside corner}  \\
\end{tabular}
\bigskip

Our main result says that if $Y$ is a two-sided 2-connected ladder with $w$ coincidental inside corners, then the set of semidualizing modules has cardinality $2^n$, where $0 \leq n \leq w$. To be specific, we prove the following; see Notation~\ref{notn181011a} for an explanation of the symbol $Z_0 \# \cdots \# Z_w$.
\medskip

{\bf Main Theorem}  [See Theorem \ref{laddertype1s}.] {\it  Let $Y= Z_0 \# \cdots \# Z_w$ be a 2-connected ladder, where each $Z_u$ is a 2-connected ladder with no coincidental inside corners. Let $R=R_2(Y)$. Then $|\mathfrak S_0(R)| = \prod_{u=0}^w|\mathfrak S_0(R_2(Z_u))| =2^{\ve_0 + \dots + \ve_w}$, where $\ve_u=0$ if $R_2(Z_u)$ is Gorenstein and $\ve_u=1$ otherwise.}

\medskip
Moreover, since Gorenstein ladder determinantal rings determined by 2-minors are completely classsified, we have:

\medskip
{\bf Corollary} [See Corollary \ref{2N}.] {\it For any $N \in \mathbb N$, there exist ladders $Y$ such that $|\mathfrak S_0(R_2(Y))| = 2^N$.  In fact, infinitely many such ladders exist.}
\medskip

The paper is organized as follows.  We begin with a Background section which, after a brief review of the relevant terms, provides some material on Bass classes.  The results on Bass classes will allow us to establish a lower bound on the number of semidualizing modules of ladder determinantal rings constructed from ladders with coincidental inside corners; see Corollary~\ref{lowerbound}.  In Section~\ref{sec3}, we prove the Main Theorem.  We begin the section by providing several base cases of ladder determinantal rings with coincidental inside corners (called corners of type 1 in \cite{CoGor}).  This is necessary, since the ladder may take many different shapes, requiring careful consideration of each possibility.  We work up to the case that $Y = Z_1 \# Z_2$, where each $Z_i$ is a two-sided ladder with no coincidental inside corner (see Proposition~\ref{Ltype1L}).  Then we are in a position to prove the Main Theorem.

%%%%%%%%%%%%%%%%%%%%%%%%%%%%%%%%%%%%%%%%%%%%%%%%%%%%%%%%%

\section{Background-Brief Recap and Material on Bass Classes} \label{sec2}

%%%%%%%%%%%%%%%%%%%%%%%%%%%%%%%%%%%%%%%%%%%%%%%%%%%%%%%%%

Citing the detailed background section in \cite{SWSeSpP1}, we provide only a brief recap of the relevant terms and facts, before proceeding to the material on Bass classes.

%~~~~~~~~~~~~~~~~~~~~~~~~~~~~~~~~~~~~~~~~~~~~~~~~~~~~~~~~~~~~~~~~~~~~~~~
\subsection{Brief Recap of Relevant Terms and Results}
%~~~~~~~~~~~~~~~~~~~~~~~~~~~~~~~~~~~~~~~~~~~~~~~~~~~~~~~~~~~~~~~~~~~~~~~

\begin{defn}  The \textbf{divisor class group} of a normal domain $R$, denoted $\Cl(R)$, is the set of isomorphism classes of rank-1 reflexive modules, or equivalently, height-1 reflexive ideals.  Denoting a module class by $[M]$, the operations $[M]+[N]=[(M\otimes_RN)^{**}]$, where $(-)^*=\Hom -R$ and $[M]-[N]=[\Hom NM]$, with additive identity $[R]$, make $\Cl(R)$ into an abelian group. 
\end{defn}

\begin{defn} A semidualizing $R$-module of finite injective dimension is a \textbf{dualizing $R$-module}. If $R$ is Cohen-Macaulay, then a dualizing module is a \textbf{canonical module}. A ring $R$ admits only {\bf trivial semidualizing modules} if 
$$\s_0(R) =\begin{cases} \{[R], [\omega_R] \} & \text{if $R$ has a dualizing module $\omega_R$};\\
\{[R]\}&\text{otherwise.}\end{cases}$$  
When we see no danger of confusion, we write $M\in\s_0(R)$ instead of $[M]\in\s_0(R)$.
A {\bf semidualizing ideal} is an ideal of  $R$ that is semidualzing as an $R$-module.
\end{defn}

\begin{fact} \label{fact:sdm} The results below will be used repeatedly.
\begin{enumerate}[(1)]
\item \label{fact:multisiso}
If $\fa, \fb$ are semidualizing ideals and $\fa \otimes_R \fb$
is semidualizing, then the multiplication map $\mu: \fa \otimes_R \fb \to
\fa \fb$ is an isomorphism by \cite[Proposition~3.3]{SSW}.

\item\label{item181011a} If $R$ is a normal domain, then $\s_0(R)\subseteq\Cl(R)$ by \cite[Proposition~3.4]{SSW}.

\item \label{fact:tensorgivesomega} If $R$ is Cohen-Macaulay with (semi)-dualizing modules $C$, $\omega_R$, respectively, then
$\Hom C{\omega_R}$ is semidualizing.  Moreover, $\Otimes{\Hom C{\omega_R}}C\cong\omega_R$ via evaluation, 
and $\Tor i{\Hom C{\omega_R}}C =0$ for all $i\geq 1$. \cite[2.11; 4.4; 4.10]{christensen:scatac}.

\item \label{twins} If $R$ is a Cohen-Macaulay normal domain with $C,\omega$ as in part~\eqref{fact:tensorgivesomega}, and $C\neq\omega_R$ are height-1 reflexive ideals,
then $\Hom C{\omega_R}$ is naturally isomorphic to a height-1 reflexive ideal $C'$, and 
$\omega_R\xla{\cong}\Otimes C{C'} \xra[\mu]{\cong}CC'$.  Thus, $[C] + [C'] = [\omega_R]$.  Conversely, if $C'$ is a height-1 reflexive ideal such that $[C] + [C'] = [\omega_R]$, then $C' \cong \Hom{C}{\omega_R}$, and hence is semidualizing.
\end{enumerate}
\end{fact}

Let $Y$ be a ladder, as described in the Introduction. The associated \textbf{ladder determinantal ring} of $t$-minors is $R_t(Y) = \sfk[Y]/I_t(Y)$, where $I_t(Y)$ is the ideal generated by the $t \times t$ minors of $X$ lying entirely in $Y$. The ring $R_t(Y)$ is known to be Cohen-Macaulay by Herzog and Trung \cite[Corollary~4.10]{HT} and a normal domain by Conca \cite[Proposition~3.3]{Co}. Let $x_{ij}$ denote the residue of $X_{ij}\in Y$ in $R_t(Y)$.

\begin{defn}\label{defn181011a} The \textbf{lower inside corners}\footnote{We use A.~Conca's \cite{Co} notation/description for lower and upper.
Thus, $(i,j) \leq (h,k)$ if and only if $i \leq h {\text{ and }} j \leq k$.  In particular, $(1,1)$ is lowest since $(1,1) \leq (h,k)$ for all $h, k$.} of $Y$ are the points $(a,b)$ such that the variables $X_{ab}, X_{a-1 b},X_{a b-1} \in Y$, but $X_{a-1 b-1}  \in X \ssm Y$; these are denoted $X_{a_i b_i}$, or simply $(a_i, b_i)$, with $1 <a_1 <\cdots<a_h <m$. For notational convenience, we also set $(a_0,b_0)=(1,n)$ and $(a_{h+1},b_{h+1})=(m,1)$.  Likewise, the \textbf{upper inside corners} of a ladder $Y$ are the points $(c,d)$ such that $X_{cd}, X_{c+1 d}, X_{c d+1} \in Y$, but $X_{c+1 d+1}  \in X \ssm Y$; these are denoted $X_{c_j d_j}$, or simply $(c_j, d_j)$, with $1 <c_1 <\cdots<c_k <m$. 
The ladder $Y$ has \textbf{coincidental corners} if $(a_i,b_i)=(c_j,d_j)$ for some $i\in\{1,\ldots,h\}$ and $j\in\{1,\ldots,k\}$ \cite[Section 1]{SWSeSpP1}.
For notational convenience, we also set $(c_0,d_0)=(1,n)$ and $(c_{k+1},d_{k+1})=(m,1)$.
A ladder $Y$ is {\bf one-sided} if it is path-connected, and $h=0$ or $k=0$. A ladder $Y$ is {\bf two-sided} if it is path-connected and $h,k>0$ \cite[Definition 1.8]{SWSeSpP1}. 
\end{defn}

\begin{ex}  In the ladders shown in the Introduction, the lower/upper inside corners of each, respectively, are (2,2)/(3,2); (4,2)/\{(2,4),(3,3)\}; and (3,2)/(3,2). The ladder $L_3$ has a coincidental inside corner at $(3,2)$.
\end{ex}

The inside corners determine the rank of the free abelian group $\Cl(R_t(Y))$.  To describe how,
we use the height-1 prime ideals of $R_2(Y)$  shown below   \cite[\S 2]{Co}:
\begin{align*}
\p_i
&=(x_{pq}\in R_2(Y)\mid\text{$p\leq c_j$ and $q\leq d_j$})&j&=1,\ldots,k
\\
\q_i
&=(x_{a_{i-1}q}\in R_2(Y))&i&=1,\ldots,h+1
\\
\q_i'
&=(x_{pb_{i-i}}\in R_2(Y))&i&=1,\ldots,h+1.
\end{align*}

\begin{fact} \label{omega} The facts below were established in \cite[\S 2]{Co}.
\begin{enumerate}[(1)]
\item The set $\{[\q_1], \ldots, [\q_{h+1}], [\p_1], \ldots, [\p_k]\}$ is a basis of $\Cl(R_2(Y))$. 
\item The canonical class is described as $[\omega_R] = \sum_{i=1}^{h+1} \lambda_i [\q_i] + \sum_{j=1}^k \delta_j[\p_j]$, where 
$\lambda_i = a_i + b_i - a_{i-1} - b_{i-1}$ for all $i = 1, \dots, h+1$ and $\delta_j = a_{i_j} + b_{i_j}-c_j-d_j$ for all $j = 1, \dots, k$, where $i_j = \min\{i :  a_i > c_j \}$. 
\item The ideals $\q_i'$ are useful for computations.  In particular, $[\q_i] + [\q_i'] + \sum_{j \in I_i} [\p_j] = 0$ for all $i = 1, \dots, h+1$, where $I_i = \{j : 1 \leq j \leq k, (a_{i-1}, b_i) \leq (c_j,d_j) \}$, \cite[(i) in Proposition 2.3]{Co}. If $I_i = \varnothing$, then $[\q_i'] = -[\q_i]$.
\end{enumerate}
\end{fact}

Finally, since we are interested in (non)-Gorenstein rings, we note:

\begin{fact} \label{Gorprop} 
The ring $R_2(Y)$ is Gorenstein if and only if $m = n$ and all inside corners $(i, j)$ of $Y$ satisfy $i + j = m + 1$ \cite[Proposition 2.5]{Co}.  In particular, if $Y$ is an $m \times n$ {\it matrix} and $m, n > 1$, then $R_2(Y)$ is Gorenstein if and only if $m = n$ \cite[Corollary~8.9]{bruns:dr}.
\end{fact}

\bigskip

%~~~~~~~~~~~~~~~~~~~~~~~~~~~~~~~~~~~~~~~~~~~~~~~~~~~~~~~~~~~~~~~~~~~~~~~
\subsection{Background-Bass Classes}
%~~~~~~~~~~~~~~~~~~~~~~~~~~~~~~~~~~~~~~~~~~~~~~~~~~~~~~~~~~~~~~~~~~~~~~~

The definition of Bass class originates with H.-B. Foxby \cite{avramov:rhafgd}.  (See also \cite{christensen:scatac}.). We present here only those few results we need.

\begin{defn} \label{Bassdef} Let $A$ be a  commutative Noetherian ring with identity, and let $M, N$ be $A$-modules such that $M \in \mathfrak S_0(A)$.  Then $N$ is in the {\bf Bass class} with respect to $M$, written $N \in \mathcal B_M(A)$, if 
\begin{enumerate}[(i)]
\item $\ext^{\geq 1}_A(M,N) = 0=\tor_{\geq 1}^A(M, (\hhom_A(M,N))$; and 
\item $M \otimes_A \hhom_A(M,N) \stackrel{\zeta_N^M}{\to} N$ is an isomorphism, where $\zeta_N^M(m \otimes \varphi) = \varphi(m)$.
\end{enumerate}
\end{defn}

\begin{ex}\label{ex181011a}
Let $A$ be a  commutative Noetherian ring with identity, and let $M,M'\in\s_0(R)$.
Then $M\in\mathcal B_M(A)$ by~\cite[Corollary~3.2.2(a)]{semidualizing moduleBook}.
Also, if $M'\in\mathcal B_M(A)$, then~\cite[Proposition 4.1.1(b)]{semidualizing moduleBook} implies that $\Hom{M}{M'}\in\s_0(A)$.
\end{ex}

\begin{lem}[\protect{\cite[Theorem~4.3]{altmann:sdmtp}}] \label{bassclasslemma} Let $R$ and $S$ be algebras finitely generated over a field $\sfk$. Set $T = R \otimes_{\sfk} S$, and let $M, M' \in \mathfrak S_0(R)$ and $N, N' \in \mathfrak S_0(S)$.  Then $M \otimes_{\sfk} N \in \mathcal B_{M' \otimes_{\sfk} N'}(T)$ if and only if $M \in \mathcal B_{M'}(R)$ and $N \in \mathcal B_{N'}(S)$.  Likewise,
$M' \otimes_{\sfk} N' \in \mathcal B_{M \otimes_{\sfk} N}(T)$ if and only if $M' \in \mathcal B_{M}(R)$ and $N' \in \mathcal B_{N}(S)$. 
\end{lem}

\begin{prop} \label{basicprops} 
Let $R$ be a standard graded ring, $R_0 = \sfk$ a field, $R = R_0[R_1]$, and $\frak m = \langle R_1 \rangle$. Let $M, M'$ be finitely-generated graded $R$-modules.  Then:
\begin{enumerate}[\rm(1)]
\item\label{basicprops1} $M \in \frak S_0(R)$ if and only if $M_{\m} \in \frak S_0(R_{\m})$;
\item\label{basicprops4} If $M \in \frak S_0(R)$, then $M \cong R$ if and only if $M_{\m} \cong R_{\m}$;
\item\label{basicprops5} $M$ is dualizing for $R$ if and only if $M_{\m}$ is dualizing for $R_{\m}$;
\item\label{basicprops6} For $M \in \frak S_0(R)$, we have $M' \in \mathcal B_M(A)$ if and only if $M'_{\m} \in \mathcal B_{M_{\m}}(A_{\m})$; 
\item\label{basicprops7} For $M, M' \in \frak S_0(R)$, we have $M \cong M'$ if and only if $M_{\m} \cong M'_{\m}$; and
\item\label{basicprops3} If $R$ is a normal domain and $\underline f \in R$ is a homogeneous $R$-regular sequence, then the map 
$\beta\colon\s_0(R)\to\s_0(R/\underline fR)$ given by $C\mapsto C/\underline fC$ is well-defined and injective.
\end{enumerate}
\end{prop}

\begin{proof}
Much of this is a variation on standard localization results. 
For instance, \cite[Proposition 2.2.3]{semidualizing moduleBook} says that a finitely generated $R$-module $C$ is
semidualizing for $R$ if and only if  for all maximal ideals $\m$ the localization $C_\m$ is semidualizing for $R_\m$.
One modifies the proof of this result, using the fact that $- \otimes_R R_{\m}$ is faithfully exact on the category of graded $R$-modules,
to establish part~\eqref{basicprops1}. 
Part~\eqref{basicprops6} is verified similarly, from the proof of~\cite[Proposition~3.5.4]{semidualizing moduleBook}.

The non-trivial implication in part~\eqref{basicprops4} follows from the fact that $M$ and $M_\m$ have the same minimal numbers of generators (over
$R$ and $R_\m$, respectively) followed by an application of~\cite[Corollary 2.1.14]{semidualizing moduleBook}.
For the non-standard implication in part~\eqref{basicprops5}, use the isomorphism $\Ext[R_\m]i{R_\m/\m R_\m}{M_\m}\cong\Ext i{R/\m}M$
to compare injective dimensions over $R$ and $R_\m$, with part~\eqref{basicprops1}.

The non-trivial implication in part~\eqref{basicprops7} merits a little more explanation. Assume that $M, M' \in \frak S_0(R)$ 
satisfy $M_{\m} \cong M'_{\m}$. Example~\ref{ex181011a} implies that $M'_\m\cong M_\m\in\mathcal B_{M_\m}(R_{\m})$
and furthermore that $\Hom{M}{M'}_\m\cong\Hom[R_\m]{M_{\m}}{M'_{\m}}\in\s_0(R_\m)$.
Since $\Hom{M}{M'}$ is finitely generated and graded, part~\eqref{basicprops1} implies that $\Hom{M}{M'}\in\s_0(R)$. 
Returning to the isomorphism $M_{\m} \cong M'_{\m}$, we conclude that 
$$\Hom{M}{M'}_\m\cong\Hom[R_\m]{M_{\m}}{M'_{\m}}\cong\Hom[R_\m]{M_{\m}}{M_{\m}}\cong R_\m$$
so $\Hom{M}{M'}\cong R$ by part~\eqref{basicprops4}. Since part~\eqref{basicprops6} implies that $M'\in\mathcal B_M(R)$, it follows by definition of 
$\mathcal B_M(R)$
that 
$$M'\cong M\otimes_R\Hom{M}{M'}\cong M\otimes_RR\cong M$$
as desired. 

For part~\eqref{basicprops3}, assume that $R$ is a normal domain and $\underline f \in R$ is a homogeneous $R$-regular sequence.
Fact~\ref{fact:sdm}\eqref{item181011a} implies that
$\s_0(R)\subseteq\Cl(R)$. Thus, since $R$ is standard graded over $\sfk$, every class of $\Cl(R)$ is represented by a graded module, so every 
semidualizing $R$-module has the structure of a graded $R$-module.
Hence, the  map $\s_0(R)\to\s_0(R/\underline fR)$ given by $C\mapsto C/\underline fC$ is well-defined by~\cite[Corollary~3.4.3]{semidualizing moduleBook}. 
To see that this map is injective\footnote{The map is not a homomorphism, as $\mathfrak S_0(-)$ has no useful group structure, so we can not just check a 
kernel condition here; see \cite[Remark~2.3.5]{semidualizing moduleBook}.}, 
suppose that $M,M' \in \mathfrak S_0(R)$ are such that $M/\underline fM \cong M'/\underline fM'$.  Then 
$M_{\m}/\underline fM_{\frak m} \cong (M/\underline fM)_{\overline{\frak m}} 
\cong (M'/\underline fM')_{\overline{\frak m}} \cong M'_{\frak m}/\underline fM'_{\frak m}$.  
By \cite[Proposition 4.2.18]{semidualizing moduleBook}, we have $M_{\frak m} \cong M'_{\frak m}$, hence by part~\eqref{basicprops7}, $M \cong M'$.  
\end{proof}

\begin{prop} \label{prop181011b} Let $R$ and $S$ be standard graded rings with $R_0 = \sfk=S_0$ a field, $R = R_0[R_1]$, and $S=S_0[S_1]$.
Set $T = R \otimes_{\sfk} S$, which is standard graded with $T_+$ maximal.
Then there an injective map $\alpha: \mathfrak S_0(R) \times \mathfrak S_0(S) \to \mathfrak S_0(T)$ defined by $\alpha([M], [N]) = [M \otimes_{\sfk} N]$.
\end{prop}

\begin{proof} 
The map $\alpha$ is well-defined by \cite[Proposition 2.3.6]{semidualizing moduleBook}.
For the injectivity of $\alpha$, let $M, M' \in \mathfrak S_0(R)$ and $N, N' \in \mathfrak S_0(S)$ such that $M \otimes_{\sfk} N \cong M' \otimes_{\sfk} N'$.  We need to show that $M \cong M'$ and $N \cong N'$.  By assumption, $M \otimes_{\sfk} N \in \mathcal B_{M' \otimes_{\sfk} N'}(T)$ and vice versa.  Thus, by Lemma \ref{bassclasslemma}, $M \in \mathcal B_{M'}(R), N \in \mathcal B_{N'}(S)$ and likewise, $M' \in \mathcal B_{M}(R), N' \in \mathcal B_{N}(S)$.  Consequently, $M_{\m} \in \mathcal B_{M'_{\m}}(R_{\m})$ and $M'_{\m} \in \mathcal B_{M_{\m}}(R_{\m})$, hence $M_{\m} \cong M'_{\m}$, and thus, $M \cong M'$, as per Proposition~\ref{basicprops}\eqref{basicprops7}.  Likewise, $N \cong N'$.  
(See also~\cite[Theorems~4.3 and~4.6]{altmann:sdmtp}.)
\end{proof}

\begin{cor} \label{lowerbound} For $t \times t$ ladder determinantal rings $R_1, R_2$ with ladders $Y_1, Y_2$, respectively, let $Z$ be the ladder constructed by identifying the lower left variable $y_1$ of $Y_1$ with the upper right variable $y_2$ of $Y_2$. Then we have
$$|\frak S_0(R_1)| \cdot |\frak S_0(R_2)| \leq |\frak S_0 (R_t(Z))| = \left|\frak S_0 \left(\frac{R_1 \otimes_{\sfk} R_2}{(y_1-y_2)} \right) \right|.$$
\end{cor}

\begin{proof}
The rings $R_1, R_2$ satisfy the assumptions of Proposition~\ref{prop181011b}.
Also, $R_1 \otimes_{\sfk} R_2$ is a normal domain since it is a ladder determinantal ring over the disconnected ladder $Y = Y_1 \cup Y_2$.  Thus, the nonzero 
homogeneous element $f = y_1 - y_2$ is regular and satisfies $R_t(Z)\cong (R_1 \otimes_{\sfk} R_2)/(y_1-y_2)$.  
The result now immediately follows from the composition of 
the injective maps $\alpha: \frak S_0(R_1) \times \frak S_0(R_2) \to \frak S_0(R_1 \otimes_{\sfk} R_2)$ and $\beta: \frak S_0(R_1\otimes_{\sfk} R_2) \to \frak S_0((R_1\otimes_{\sfk} R_2)/f)$ from Propositions~\ref{prop181011b} and~\ref{basicprops}\eqref{basicprops3}. 
\end{proof}
\bigskip

%%%%%%%%%%%%%%%%%%%%%%%%%%%%%%%%%%%%%%%%%%%%%%%%%%

\section{Proof of Main Theorem} \label{sec3}

%%%%%%%%%%%%%%%%%%%%%%%%%%%%%%%%%%%%%%%%%%%%%%%%%%

We will prove our main result, Theorem~\ref{laddertype1s}, in this section through a series of inductions.  Because many of the arguments proceed in a similar manner, in certain cases only highlights are provided.  We begin with some additional notation and an example that will be carried throughout the section.
%We recall that $2 \times 2$ minors of a ladder $Y$ are special in the sense that $R_2(Y)$ is an algebra with straightening laws, or ASL, on the poset $Y$, as per \cite[p.~121]{Co}.  (However, $R_t(Y)$, for $t > 2$, is not.)

%~~~~~~~~~~~~~~~~~~~~~~~~~~~~~~~~~~~~~~~~~~~~~~~~~~~~~~~~~~~~~~~~~~~~~~~
\subsection{Preliminaries-notation}
%~~~~~~~~~~~~~~~~~~~~~~~~~~~~~~~~~~~~~~~~~~~~~~~~~~~~~~~~~~~~~~~~~~~~~~~

\begin{comment}
\begin{defn} An inside corner of {\bf type 1} occurs in a two-sided $t$-connected ladder when, for some $i = 1, \dots, h$ and some $j = 1, \dots, k$, it holds that $c_j- a_i \leq t-2$ and $d_j - b_i \leq t-2$.  
\end{defn}

In the $2 \times 2$ case, these corners occur exactly when an upper inside corner and a lower inside corner coincide; in the terminology of \cite{Co}, \cite{SWSeSpP1}, for some $i, j$, $(a_i, b_i) = (c_j, d_j)$.  The situation where $a_i - c_j \leq t-2$ or  $b_i - d_j \leq t-2$ will result in a $t$-disconnected ladder; i.e., a ladder that is not connected with respect to the $t \times t$ minors.
\end{comment}

\begin{notn}\label{notn181011a} Let $Z_0, Z_1$ be ladders in matrices of minimal size $m_0 \times n_0$, $m_1 \times n_1$, respectively. We define $Z_0 \# Z_1$ to be the ladder with a coincidental inside corner formed by identifying the variable $X_{m_0 1}$ of $Z_0$ with the variable $X_{1 n_1}$ of $Z_1$.  We repeat this process to get $Z_0 \# Z_1 \# \cdots \# Z_w$.
\end{notn}

\begin{ex} \label{type1} The ladder $L_3$ in the Introduction is $Z_0 \# Z_1$ of two $3 \times 2$ matrices $Z_0$, $Z_1$; i.e., ladders with no inside corners.  If the elements of $Z_0$ and $Z_1$ are indexed as below, then we identify $X_{32}$ with $X'_{32}$.

\begin{table}[h]
  \begin{center}
    \begin{tabular}{c c c c c c c c c c c c c c c c c c c c c}
       \!\!\!\!\!$X_{12}$ & \!\!\!\!\!$X_{13}$ & & & & & \!\!\!\!\!$X'_{31}$ & \!\!\!\!\!$X'_{32}$ \\
       \!\!\!\!\!$X_{22}$ & \!\!\!\!\!$X_{23}$ & & & & & \!\!\!\!\!$X'_{41}$ & \!\!\!\!\!$X'_{42}$ & & & \\
       \!\!\!\!\!$X_{32}$ & \!\!\!\!\!$X_{33}$ & & & & & \!\!\!\!\!$X'_{51}$ & \!\!\!\!\!$X'_{52}$ \\
         \end{tabular}
  \end{center}
\end{table}

\hskip1.45in $Z_0$ \hskip.9in  $Z_1$ 

\medskip
Then $\displaystyle{R_2(L_3) = R_2(Z_0 \# Z_1) =  \frac{R_2(Z_0) \otimes R_2(Z_1)}{(x_{32} - x'_{32})}}$. In terms of Corollary~\ref{lowerbound}, we have $|\frak S_0(R_2(Z_0))| \cdot |\frak S_0(R_2(Z_1))| \leq |\frak S_0(R_2(Z_0 \# Z_1 )|.$ 
\end{ex}
\medskip

\begin{notn} \label{notn:ideals}
  Let $Z_0, Z_1, \dots,Z_w$ be ladders with no coincidental inside corners, and let $Z=Z_0\# Z_1 \# \cdots \# Z_w$.
  We will use double indices to label the corners of $Z$ and the generators of
  $\Cl(R_2(Z))$. The ladder $Z_u$ will have $h_u$ lower inside corners and
  $k_u$ upper inside corners. The inside corners of $Z$ which are inside corners of $Z_u$ will be written as
  $(a_{ui},b_{ui}), (c_{uj},d_{uj})$, for $1 \leq i \leq h_u$ and $1 \leq j \leq k_u$. Additionally, there are inside corners of $Z$, which are not inside corners of any $Z_u$, but which are  
   variables coincidental to some $Z_u$ and $Z_{u+1}$.  In particular, for all $0\leq u<w$ we have
  $(a_{u,h_u+1},b_{u,h_u+1})=(c_{u,k_u+1},d_{u,k_u+1})=
  (a_{u+1,0},b_{u+1,0})=(c_{u+1,0},d_{u+1,0})$.  Similarly, we label the ideals
  of $Z$ as $\q_{ui},\p_{uj}$. Moreover, we write the ideals that contain the
  variable at the $u$-th coincidental inside corner, where $1\leq u \leq w$, as $\q_{u1}$ and $\p_{u0}$.
  That is,
  \begin{align*}
    \p_{uj} &=(x_{pq}\in R_2(Z) \mid p\leq c_{uj} \text{ and } q\leq d_{uj})
    && \text{for all } 1 \leq j \leq k_u \text{ and } 0 \leq u \leq w,\\
    \p_{u0} &=(x_{pq}\in R_2(Z) \mid p\leq c_{u0} \text{ and } q\leq d_{u0})
    && \text{for all } 1 \leq u \leq w, \text{ and}\\
    \q_{ui} &=(x_{a_{\{u,i-1\}},q}\in R_2(Y) \mid q \in \bbn)
    && \text{for all } 1 \leq i \leq h_u+1 \text{ and } 0 \leq u \leq w.
  \end{align*}
  We will also use $\q_{ui},\p_{uj}$ to denote the \emph{restrictions} of these
  ideals in $R_2(Z_u)$ for $j \neq 0$. On the other hand, we will identify
  $[\omega_{R_2(Z_u)}] \in \Cl(R_2(Z_u))$ with its \emph{image} in
  $\Cl(R_2(Z))$. That is, we write
  \begin{align*}
    [\omega_{R_2(Z_u)}] &= \sum_{i=1}^{h_u+1} \gl_{ui} [\q_{ui}]
    + \sum_{j=1}^{k_u} \gd_{uj} [\p_{uj}]
    \text{ in } \Cl(R_2(Z_u)), \text{ but}\\
    [\omega_{R_2(Z_u)}] &= \gl_{u1} ([\p_{u0}] + [\q_{u1}])
    + \sum_{i=2}^{h_u+1} \gl_{ui} [\q_{ui}] + \sum_{j=1}^{k_u} \gd_{uj} [\p_{uj}]
    \text{ in } \Cl(R_2(Z)).
  \end{align*}
  Note that in $\Cl(R_2(Z))$, the class $[\p_{u0}]+[\q_{u1}] = [\p_{u0} \cap \q_{u1}]$
  is the image of $[\q_{u1}] \in \Cl(R_2(Z_u))$ by \cite[Lemma 2.2]{SWSeSpP1}.
  We then have $[\omega_{R_2(Z)}] = [\omega_{R_2(Z_0)}]
  + \dots + [\omega_{R_2(Z_w)}]$ in $\Cl(R_2(Z))$.
\end{notn}

\noindent  {\bf Example~\ref{type1} (continued).}  Recall $L_3 = Z_0\# Z_1$ with
  corners $(a_{00},b_{00})=(c_{00},d_{00})$ $=(1,3)$,
  $(a_{01},b_{01})=(a_{10},b_{10})=(c_{01},d_{01})=(c_{10},d_{10})
  =(3,2)$ and $(a_{11},b_{11})=(c_{11},d_{11})=(5,1)$. The ring
  $R_2(Z)$ has ideals
  $\q_{01}=(x_{12},x_{13})$, $\q_{11}=(x_{31},x_{32},x_{33})$
  and $\p_{10}=(x_{12},x_{22},x_{31},x_{32})$.
  The ring $R_2(Z_0)$ has ideal $\q_{01}=(x_{12},x_{13})$
  and the ring $R_2(Z_1)$ has ideal $\q_{11}=(x'_{31},x'_{32})$.

  We identify $[\omega_{R_2(Z_0)}]=[(x_{12},x_{13})]=[\q_{01}] \in \Cl(R_2(Z_0))$
  with $[\q_{01}] \in \Cl(R_2(L_3))$. We identify $[\omega_{R_2(Z_1)}]
  =[(x'_{31},x'_{32})] = [\q_{11}] \in \Cl(R_2(L_3))$ with $[\q_{11}]+[\p_{10}] = [\q_{11} \cap \p_{10}]
  = [(x_{31},x_{32},x_{33}) \cap (x_{12},x_{22},x_{31},x_{32})]
  = [(x_{31},x_{32})] \in \Cl(R_2(L_3))$. With such identification, we have $[\omega_{R_2(L_3)}]
  = [\q_{01}]+[\q_{11}]+[\p_{10}] = [\omega_{R_2(Z_1)}]+[\omega_{R_2(Z_2)}]$.

\begin{notn} \label{notn:ladder}
  When we are considering a ladder $Y$ and would like to discuss a new related
  ladder, we will use the notation $Y^{\bullet}, Y^{\dag}, \ti{Y}$, etc., to denote the new ladders.
 The notation $R^{\bullet},R^{\dag}$, etc.,\ will always denote the associated ladder determinantal ring $R_2(Y^{\bullet}),R_2(Y^{\dag})$,  respectively.
\end{notn}

\begin{defn} (\cite[Definition 3.4]{SWSeSpP1})
  The \textbf{antitranspose} of a ladder $Y$ is the ladder $\ti{Y}$
  obtained by \textbf{antitransposing} the ladder $Y$, i.e.\ reflecting $Y$
  along the antidiagonal,
  so that $\ti{Y}_{ij}= X_{a_{h+1}-j+a_0,b_0-i+b_{h+1}}$.
  The ladder $\ti{Y}$ has corners $(\ti{a}_0,\ti{b}_0) = (b_{h+1},a_{h+1})$,
  $(\ti{a}_1,\ti{b}_1) = (b_0-d_1+b_{h+1},a_{h+1}-c_1+a_0)$, \ldots,
  $(\ti{a}_k,\ti{b}_k) = (b_0-d_k+b_{h+1},a_{h+1}-c_k+a_0)$,
  $(\ti{a}_{k+1},\ti{b}_{k+1}) = (b_0,a_0)$,
  $(\ti{c}_1,\ti{d}_1) = (b_0-b_1+b_{h+1},a_{h+1}-a_1+a_0)$, \ldots,
  $(\ti{c}_h,\ti{d}_h) = (b_0-b_h+b_{h+1},a_{h+1}-a_h+a_0)$.
\end{defn}

%~~~~~~~~~~~~~~~~~~~~~~~~~~~~~~~~~~~~~~~~~~~~~~~~~~~~~~~~~~~~~~~~~~~~~~~
\subsection{Base cases}
%~~~~~~~~~~~~~~~~~~~~~~~~~~~~~~~~~~~~~~~~~~~~~~~~~~~~~~~~~~~~~~~~~~~~~~~

We begin establishing the main result by proving some base cases.  In each of these statements (3.6-3.11), the number of semidualizing modules of $R = R_2(Y)$ is either 1, 2, or 4; we are setting $Y = Z_0 \# Z_1$, hence $\mathfrak S_0(R) = 2^{\ve_0 + \ve_1}$, where $\ve_i=0$ if $R_2(Z_i)$ is Gorenstein and $\ve_i=1$ otherwise.

\begin{prop} \label{onematrixtype1}
  Let $Y = Z_0 \# Z_1$ be a 2-connected ladder with exactly one coincidental inside corner, where $Z_0, Z_1$ are matrices of indeterminates, as shown below. Let $R = R_2(Y)$. Then $\mathfrak S_0(R)=\{[R],[\omega_{R_2(Z_0)}],[\omega_{R_2(Z_1)}],[\omega_R]\}$, where $[R]$ is the 0 class and $[\omega_R]=[\omega_{R_2(Z_0)}]+[\omega_{R_2(Z_1)}]$. In particular,
$|\mathfrak S_0(R)| = |\mathfrak S_0(R_2(Z_0))| \cdot |\mathfrak S_0(R_2(Z_1))|$.  

\begin{center}
    \begin{tabular}{c c c c c c c c}
        &  & & $X_{1,b_{01}}$ & $X_{1,b_{01}+1}$ & $\cdots$ &  & $X_{1,n}$\\
         &  & &  $\vdots$           &      $\vdots$             &        $Z_0$        &                         & $\vdots$\\
          &  & & $X_{a_{01}-1,b_{01}}$                & $X_{a_{01}-1,b_{01}+1}$                     &      $\cdots$          &                        & $X_{a_{01}-1,n}$\\
        $X_{a_{10},1}$ & $\cdots$ & $X_{a_{10}, b_{10}-1}$ & $X_{a_{10},b_{10}}$ & $X_{a_{01},b_{01}+1}$ & $\cdots$ &  & $X_{a_{01},n}$\\
       $X_{a_{10}+1,1}$ & $\cdots$  & $X_{a_{10}+1, b_{10}-1}$ & $X_{a_{10}+1,b_{10}}$\\
       $\vdots$   &     $Z_1$    &  $\vdots$  & $\vdots$  & & \multicolumn{3}{l}{$(a_{10},b_{10})=(a_{01},b_{01})$}\\
       $X_{m,1}$ & $\cdots$ & $X_{m, b_{10}-1}$ & $X_{m,b_{10}}$ & & \multicolumn{3}{l}{$=(c_{10},d_{10})=(c_{01},d_{01})$}
    \end{tabular}
  \end{center}
\end{prop}

\begin{proof}  Let $Y$ be the ladder shown above, where $2\leq a_{10} \leq m-1, 2\leq b_{10} \leq n-1$. This ladder is 2-connected and we have $(a_{00},b_{00}) = (c_{00},d_{00})= (1,n)$, 
$(a_{01},b_{01}) = (a_{10},b_{10}) = (c_{01},d_{01}) = (c_{10},d_{10})$, and $(a_{11},b_{11}) =  (c_{11},d_{11}) =(m,1)$.  With $\frakq_{01}, \frakq_{11},$ and $\frakp_{10}$ the ideals shown below, the class group of $R$ is $\Cl(R) \cong \bbz[\q_{01}] \oplus \bbz[\q_{11}] \oplus \bbz [\p_{10}]$, where
\begin{align*}
  \frakq_{01} &= (x_{1,b_{01}}, x_{1,b_{01}+1}, \dots, x_{1,n}), \\
  \frakq_{11} &= (x_{a_{10},1},\dots, x_{a_{10},b_{10}},\dots, x_{a_{10},n}), \text{ and}\\
  \frakp_{10} &= (x_{1,b_{01}}, x_{2,b_{01}}, \dots, x_{a_{01},b_{01}}, x_{a_{10}, 1}, x_{a_{10},2}, \dots, x_{a_{10},b_{10}-1}).
\end{align*}
The canonical class of $R$ is $[\omega_R] = \lambda_{01}[\frakq_{01}] + \lambda_{11}[\frakq_{11}]  + \lambda_{11}[\frakp_{10}]$, where
$\lambda_{01} = a_{01}+b_{01}-1-n$, $\lambda_{11} = m+1-a_{10}-b_{10}$ and $\delta_{10} = \lambda_{11}$.

The proof will proceed by inverting variables in $Y$.
We will let $C_1,C_2,\dots$ denote possible semidualizing modules of $R$.

\begin{step}
  First, let $Y^{\bullet}$ be the ladder obtained by deleting rows $a_{00}, a_{00}+1,\dots,a_{01}-1$ and columns $b_{01}+1, b_{01}+2,\dots, b_{00}$ of $Y$; that is, $Y^{\bullet} = Z_1$.
  Invert $x_{1,b_{01}}$ in $R$ and
  let $\rho^{\bullet}$ be the composition of the following natural surjections:
  \[
    \Cl(R) \to \Cl(R_{x_{1,b_{01}}}) \xrightarrow{\cong} \Cl(R^{\bullet}).
  \]

  In particular, $\Cl(R^{\bullet}) \cong \bbz[\q_{11}]$, where (the new) $\q_{11}$ is the ideal generated by the (images in $R^{\bullet}$ of the) variables in the first row of $Y^{\bullet}$, by \cite[Corollary 8.4]{bruns:dr}, and $[\omega_{R^{\bullet}}] = \lambda_{11}[\q_{11}]$ by \cite[(7.10),(8.8)]{bruns:dr}.
  
    Under the natural map $\rho^{\bullet} \colon \Cl(R) \to \Cl(R^{\bullet})$, we have
    $\rho^{\bullet}([\q_{01}])=0, \rho^{\bullet}([\p_{10}]) = 0$ and $\rho^{\bullet} ([\q_{11}]) = [\q_{11}]$,
    so $\ker(\rho^{\bullet}) = \bbz[\q_{01}] \oplus \bbz [\p_{10}]$. The determinantal ring $R^{\bullet}$ has semidualizing modules $R^{\bullet}$ and $\omega_{R^{\bullet}}$ only.  Since the localization of a semidualizing module is also a semidualizing
  module, the only possible semidualizing modules of $R$ are in
  $\vf^{-1}([R^{\bullet}]) = \bbz[\q_{01}] \oplus \bbz [\p_{10}]$ or $\vf^{-1}([\omega_{R^{\bullet}}]) = \bbz[\q_{01}] \oplus \bbz [\p_{10}]
  + \lambda_{11}[\q_{11}]$. Thus, the possible semidualizing modules
  of $R$ are $[C_1]=r[\q_{01}]+s[\p_{10}]$ and $[C_2]=u[\q_{01}] + v[\p_{10}] + \lambda_{11}[\q_{11}](=u[\q_{01}] + v[\p_{10}]+[\omega_{R_2(L)}]-\gl_{11}[\p_{10}])$, where $r,s,u,v \in \bbz$.
\end{step}
  
\begin{step}
  Next, obtain a ladder $Y^{\dag}$ by deleting rows $1,\dots,a_{01}-1$ and columns $b_{01}+1,\dots, n$ of $Y$; in fact, $Y^{\dag} = Z_1$.  Invert $x_{a_{01},n}$ in $R$ and
  let $\rho^{\dag}$ be the composition of the following natural surjections:
  \[
    \Cl(R) \to \Cl(R_{x_{a_{01},n}}) \xrightarrow{\cong} \Cl(R^{\dag}).
  \]
    Under the natural map $\rho^{\dag} \colon \Cl(R) \to \Cl(R^{\dag})$, we have
    $[\q_{01}], [\q_{11}] \mapsto 0$ and $[\p_{10}] \mapsto [\q_{11}]$ (the new $\q_{11}$).
  Again $\s_0 (R^{\dag})=\{[R^{\dag}],[\omega_{R^{\dag}}]\}$, where $[\omega_{R^{\dag}}] = \lambda_{11}[\q_{11}]$. 

  To determine the semidualizing modules of $R$, consider the possible
  images of $[C_1],[C_2]$ under $\rho^{\dag}$:
   $$\rho^{\dag}(r[\q_{01}]+s[\p_{10}]) = 0 \Rightarrow s = 0 \quad {\text{and}} \quad \rho^{\dag}(r[\q_{01}]+s[\p_{10}]) = \lambda_{11}[\q_{11}]  \Rightarrow s = \lambda_{11},$$
   and similarly $v=0$ or $\lambda_{11}$.
   Hence, the possible semidualizing modules of $R$ are $[C_3]=r[\q_{01}]$, $[C_4]=r[\q_{01}]+\lambda_{11}[\p_{10}]$,
  $[C_5]=u[\q_{01}] + \lambda_{11}[\q_{11}]$ and $[C_6]=u[\q_{01}] + \lambda_{11}[\p_{10}] + \lambda_{11}[\q_{11}]$.
\end{step}
    
\begin{step}
  Thirdly, obtain $Y^{\bullet\bullet}$ by deleting rows
    $a_{10}+1,\dots,m$ and columns $1,\dots,b_{10}-1$ of $Y$;  that is, $Y^{\bullet\bullet} = Z_0$.
  Invert $x_{a_{10},1}$ in $R$.
  Under the natural map $\rho^{\bullet\bullet} \colon \Cl(R) \to \Cl(R^{\bullet\bullet})$, we have $[\q_{11}], [\p_{10}] \mapsto 0$ and $[\q_{01}] \mapsto [\q_{01}]$.
  Since $R^{\bullet\bullet}$ is a determinantal ring, we know
    that $\s_0 (R^{\bullet\bullet})=\{[R^{\bullet\bullet}],[\omega_{R^{\bullet\bullet}}]\}$, where $[\omega_{R^{\bullet\bullet}}] = \lambda_{01}[\q_{01}]$.

  If $\rho^{\bullet\bullet}([C_3])=r[\q_{01}]=0$, then $r=0$, and $[C_3]=0$ is a trivial semidualizing module of $R$. If $\rho^{\bullet\bullet}([C_3])=\lambda_{01}[\q_{01}]$, then
  $r=\lambda_{01}$. Doing the same for $C_4,C_5,C_6$, we get the following possible nontrivial semidualizing modules of $R$.
  \begin{align*}
    [C_7] &= \lambda_{01}[\q_{01}]=[\omega_{R_2(Z_0)}]\\ 
    [C_8] &= \lambda_{11}[\p_{10}]\\
    [C_9] &= \lambda_{01}[\q_{01}] + \lambda_{11}[\p_{10}]=[\omega_R]-[C_{10}]\\
    [C_{10}] &= \lambda_{11}[\q_{11}]\\
    [C_{11}] &= \lambda_{01}[\q_{01}] + \lambda_{11}[\q_{11}]=[\omega_R]-[C_8]\\
    [C_{12}] &= \lambda_{11}[\p_{10}] + \lambda_{11}[\q_{11}]=[\omega_{R_2(Z_1)}]
   \end{align*}

    Hence, it remains to show that $C_8,C_9,C_{10},C_{11}$ can not be nontrivial semidualizing modules of $R$.
\end{step}
    
\begin{step}
  Fourthly, obtain the ladder $Y^{\dag\dag}$ by deleting rows
    $a_{10}+1,\dots,m$ and columns $1,\dots,b_{10}-1$ of $Y$; in fact, $Y^{\dag\dag} = Z_0$.  Invert $x_{m,b_{10}}$ in $R$.
      Under the natural map $\rho^{\dag\dag} \colon \Cl(R) \to \Cl(R^{\dag\dag})$, we have
    $[\q_{01}] \mapsto [\q_{01}]$, $[\q_{11}] \mapsto [(x_{a_{01}, b_{01}}, \dots, x_{a_{01},n})] = [\q_{01}]$, and $[\p_{10}] \mapsto [(x_{1, b_{01}}, x_{2, b_{01}}, \dots, x_{a_{01},b_{01}})]=-[\q_{01}]$.
    Again, $\s_0 (R^{\dag\dag}) = \{[R^{\dag\dag}],[\omega_{R^{\dag\dag}}]\}$, where $[\omega_{R^{\dag\dag}}] = \lambda_{01}[\q_{01}]$.

  We can now show that the modules $C_8,C_9,C_{10},C_{11}$ can not be
  nontrivial semidualizing modules of $R$. If $\rho^{\dag\dag}([C_8])
  =-\lambda_{11}[\q_{01}]=0$ (equivalently, $\rho^{\dag\dag}([C_{11}])
  =[\omega_{R^{\dag\dag}}]$), then $\lambda_{11}=0$, so $[C_8]=0$
  is a trivial semidualizing module.
  If $\rho^{\dag\dag}([C_8])=[\omega_{R^{\dag\dag}}]=
  \lambda_{01}[\q_{01}]$ (equivalently, $\rho^{\dag\dag}([C_{11}])=0$),
  then $\lambda_{11}=-\lambda_{01}$. Similarly, if $\rho^{\dag\dag}
  ([C_{10}])=\lambda_{11}[\q_{01}]=0$ or $[\omega_{R^{\dag\dag}}]$
  (equivalently, $\rho^{\dag\dag}([C_9])=[\omega_{R^{\dag\dag}}]$ or 0
  respectively), then $\lambda_{11}=0$ or $\lambda_{01}$ respectively.
  So we only need to show that $\lambda_{11}=\pm \lambda_{01} \neq 0$
  leads to a contradiction.

  \begin{case} \label{case:negative}
    $\lambda_{11}=-\lambda_{01}>0$. We have $[C_8]=\lambda_{11}[\p_{10}]$ and
    $[C_{11}]=-\lambda_{11}[\q_{01}] + \lambda_{11}[\q_{11}]$. By Fact \ref{omega}(3),
    we have $[C_{11}]=\lambda_{11}([\q'_{01}]+[\p_{10}]+[\q_{11}])$.
    By \cite[Lemma~2.1]{SWSeSpP1},
    \begin{align*}
      [C_8] &= [\p_{10}^{\lambda_{11}}] \text{ and}\\
      [C_{11}] &= [(\q'_{01})^{\lambda_{11}} \cap \p_{10}^{\lambda_{11}} \cap
      \q_{11}^{\lambda_{11}}].
    \end{align*}
    Let us identify $C_8$ with the ideal $\p_{10}^{\lambda_{11}}$, and likewise for $C_{11}$. 
    Then under the multiplication map $\mu \colon C_8 \otimes C_{11} \to C_8C_{11}$,
    we have
    \[
      \mu(x_{a_{10},b_{10}}^{\lambda_{11}} \otimes
             x_{1,b_{01}}x_{a_{10}, 1}x_{a_{10},b_{10}}^{\lambda_{11}-1})
      = \mu(x_{a_{10},1}x_{1,b_{01}}x_{a_{10},b_{10}}^{\lambda_{11}-1}
         \otimes x_{a_{10},b_{10}}^{\lambda_{11}}).
    \]
    Hence $\mu$ is not injective, contradicting Fact~\ref{fact:sdm} \eqref{fact:multisiso},
    so $C_8,C_{11}$ are not semidualizing modules.
    Since $\lambda_{11} \neq \lambda_{01}$ in this case, the modules $C_9,C_{10}$
    are not semidualizing either, so the only remaining possible classes of
    nontrivial semidualizing modules are $[C_7]=[\omega_{R_2(Z_0)}]$ and
   $[C_{12}]=[\omega_{R_2(Z_1)}]$.
  \end{case}

  \begin{case} \label{case:positive}
    $\lambda_{11}=\lambda_{01}>0$. By \cite[Lemma~2.1]{SWSeSpP1}, we have
    \begin{align*}
      [C_9] &= \gl_{11}[\q_{01}]+\gl_{11}[\p_{10}]
      =[\q_{01}^{\gl_{11}} \cap \p_{10}^{\gl_{11}}] \text{ and}\\
      [C_{10}] &= [\q_{11}^{\gl_{11}}].
    \end{align*}
    As in the previous case, we identify $C_9,C_{10}$ with the corresponding ideals on the right.
    Then under the multiplication map $\mu \colon C_9 \otimes C_{10} \to C_9C_{10}$,
    we have
    \begin{align*}
      \mu(x_{1,b_{01}}^{\gl_{11}-1} x_{1,n} x_{a_{10},1} \otimes x_{a_{01},b_{01}} x_{a_{01},n}^{\gl_{11}-1})
      &= x_{a_{10},1} x_{1,b_{01}}^{\gl_{11}-1} x_{1,n} x_{a_{01},b_{01}} x_{a_{01},n}^{\gl_{11}-1}\\
      &= x_{a_{10},1} x_{1,b_{01}}^{\gl_{11}} x_{a_{01},n}^{\gl_{11}}\\
      &= \mu(x_{a_{10},1}x_{1,b_{01}}^{\gl_{11}} \otimes x_{a_{01},n}^{\gl_{11}}).
    \end{align*}
    Hence $\mu$ is not injective, and we reach the same conclusion as in
    Case~\ref{case:negative}.
  \end{case}

  \begin{case}
    $\lambda_{01}=-\lambda_{11}>0$. In this case, we take the antitranspose
    $\ti{Y}$ of $Y$. The coincidental inside corner
    of $\ti{Y}$ is at $(n+1-b_{01},m+1-a_{01})=(n+1-b_{10},m+1-a_{10})$.
    For $\ti{Y}$, we have $\ti{\gl}_{01}
    = (n+1-b_{01})+(m+1-a_{01})-(m+1)=n+1-a_{01}-b_{01}=-\gl_{01}$ and
    $\ti{\gl}_{11}=(n+1)-(n+1-b_{10})-(m+1-a_{10})=a_{10}+b_{10}-m-1=-\gl_{11}$.
    Then $\ti{\gl}_{11}=-\ti{\gl}_{01}>0$, and we can use Case~\ref{case:negative}.
  \end{case}

  \begin{case}
    $\gl_{01}=\gl_{11}<0$. Then we antitranspose $Y$ and use Case~\ref{case:positive}.
  \end{case}
\end{step}

In summary, we have established that $\mathfrak S_0(R)\subseteq\{[R],[\omega_{R_2(Z_0)}],
[\omega_{R_2(Z_1)}],[\omega_R]\}$, where $[\omega_R]=[\omega_{R_2(Z_0)}]+[\omega_{R_2(Z_1)}]$, 
and hence $|\mathfrak S_0(R)| \leq |\mathfrak S_0(R_2(Z_0))| \cdot |\mathfrak S_0(R_2(Z_1))|$.
The reverse inequality is given by Corollary~\ref{lowerbound}, which completes the proof.
\end{proof}

\begin{lem} \label{onesidedtype1} Let $Y = Z_0 \# Z_1$ be a 2-connected ladder, where $Z_0$ is a matrix of indeterminates and $Z_1$ is a one-sided ladder. 
Let $R=R_2(Y)$. Then $\mathfrak S_0(R)=\{[R],[\omega_{R_2(Z_0)}],[\omega_{R_2(Z_1)}],[\omega_R]\}$.
In particular, $|\mathfrak S_0(R)| = |\mathfrak S_0(R_2(Z_0))| \cdot |\mathfrak S_0(R_2(Z_1))|$.
\end{lem}
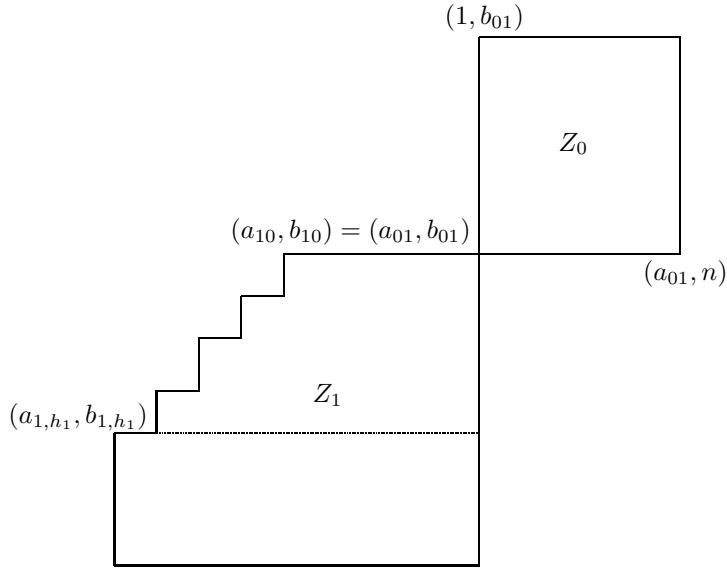
\smallskip
\begin{center}  
  \begin{picture}(204,204)
    \put(0,0){\line(0,1){50}}
    \put(0,50){\line(1,0){16}}
    \put(16,50){\line(0,1){16}}
     \put(16,66){\line(1,0){16}}
    \put(32,66){\line(0,1){20}}
    \put(32,86){\line(1,0){16}}
    \put(48,86){\line(0,1){16}}
    \put(48,102){\line(1,0){16}}
    \put(64,102){\line(0,1){16}}
    \put(64,118){\line(1,0){150}}
    \put(0,0){\line(1,0){138}}
    \put(138,0){\line(0,1){200}}
    \put(138, 200){\line(1,0){76}}
    \put(214,118){\line(0,1){82}}
    \put(44,124){$(a_{10},b_{10})=(a_{01},b_{01})$}
    \put(200,108){$(a_{01},n)$}
    \put(125,205){$(1, b_{01})$}
    \put(-40,54){$(a_{1,h_1},b_{1,h_1})$}
    \put(168,157){$Z_0$}
    \put(75,62){$Z_1$}
     \multiput(16,50)(1,0){122}{\line(1,0){.25}}
 \end{picture}
\end{center}

\begin{proof} Since $Z_0$ is a matrix, $h_0=k_0=0$. Now if $h_1=0$ for the ladder $Z_1$, then the ladder $Y$ can be antitransposed to obtain a one-sided ladder $\ti{Z_1}$ with $k_1=0$. Hence we may assume that $k_1=0$ for the ladder $Z_1$, and that $Z_0 \# Z_1$ takes the shape above. Set $Y = Z_0 \# Z_1$ and $R = R_2(Y)$. The class group of $R$ is $\Cl(R) \cong \bbz[\q_{01}] \oplus \bbz[\q_{11}] \oplus \bbz[\q_{12}] \oplus \cdots \oplus \bbz[\q_{1,h_1+1}] \oplus \bbz [\p_{10}]$; i.e., it is free of rank $h_1+3$.
The canonical class of $R$ is
\[
  [\omega_R] = \gl_{01}[\q_{01}] + \sum_{i=1}^{h_1+1} \gl_{1i} [\frak q_{1i}] + \delta_{10} [\frakp_{10}],
\]
where $\gl_{01}=a_{01}+b_{01}-1-n$, $\gl_{1i}=a_{1,i}+b_{1,i}-a_{1,i-1}-b_{1,i-1}$ and $\delta_{10}=\gl_{11}$.

The Lemma is proved by induction on $h_1$.  The base case of $h_1 = 0$ is given by Proposition~\ref{onematrixtype1}.  Thus, let $h_1 > 0$ and assume that the Lemma holds for any 2-connected ladder $\ti{Z_0} \# \ti{Z_1}$, where $\ti{Z_0}$ is a matrix and $\ti{Z_1}$ is a one-sided ladder with $h_1-1$ lower inside corners.  As before, $C_1,C_2$ will denote possible semidualizing modules of $R$.

\setcounter{step}{0}
\begin{step}
  Obtain the ladder $Y^{\bullet}$ by deleting rows $1,\dots,a_{01}-1$ and columns $b_{01}+1,\dots, b_{00}$ of $Y$; that is, $Y^{\bullet} = Z_1$.  Invert $x_{1, b_{01}}$ in $R$.  Under the natural map $\rho^{\bullet} \colon \Cl(R) \to \Cl(R^{\bullet})$, where $R^{\bullet} = R_2(Z_1)$, we have $[\q_{01}], [\p_{10}] \mapsto 0$ and $[\q_{1i}] \mapsto [\q_{1i}]$ for all $1 \leq i \leq h_1+1$. We know $\Cl(R^{\bullet}) \cong \bbz^{h_1+1}$, generated by the ideals $\q_{1i}$ for $1 \leq i \leq h_1+1$, so $\ker(\rho^{\bullet}) = \bbz[\q_{01}] \oplus \bbz [\p_{10}]$. By the One-Sided Ladder Theorem \cite{SWSeSpP1}, $\s_0(R^{\bullet}) = \{0, [\omega_{R^{\bullet}}]\}$, where $[\omega_{R^{\bullet}}] = \sum_{i=1}^{h_1+1} \gl_{1i} [\frak q_{1i}]$. Hence the possible classes of semidualizing modules of $R$ are $[C_1] = r[\q_{01}]+s[\p_{10}]$ and $[C_2]=u[\q_{01}]+v[\p_{10}] + \sum_{i=1}^{h_1+1} \gl_{1i} [\frak q_{1i}](= u[\q_{01}]+v[\p_{10}]+[\omega_{R^{\bullet}}]-\gl_{11}[\p_{10}])$, where $r,s,u,v \in \bbz$.
\end{step}
 
\begin{step}
  Next, obtain $Y^{\bullet\bullet}$ by deleting rows $a_{1,h_1}+1,\dots, a_{1,h_1+1}$ and columns $1,\dots,b_{1,h_1}-1$ of $Y$. Then $Y^{\bullet\bullet}=Z_0 \#Z_1^{\bullet\bullet}$, where $Z_1^{\bullet\bullet}$ is a one-sided ladder with one fewer inside corner than $Z_1$.  Invert $x_{a_{1,h_1}, 1}$  in $R$.
Under the natural map $\rho^{\bullet\bullet} \colon \Cl(R) \to \Cl(R^{\bullet\bullet})$, we have $[\q_{1,h_1+1}] \mapsto 0$, $[\q_{01}]\mapsto [\q_{01}]$, $[\p_{10}]\mapsto [\p_{10}]$ and $[\q_{1i}] \mapsto [\q_{1i}]$ for all $1\leq i \leq h_1$.  Since $Z_1^{\bullet\bullet}$ has $h_1-1$ lower inside corners, the induction hypothesis gives $\s_0 (R^{\bullet\bullet})=\{0,[\omega_{R_2(Z_0)}],[\omega_{R_2(Z_1^{\bullet\bullet})}],[\omega_{R^{\bullet\bullet}}]\}$. We consider several cases.

\setcounter{case}{0}
\begin{case}        
If $\rho^{\bullet\bullet}([C_1])=0$, then $r=s=0$, and $[C_1]=0$.

If $\rho^{\bullet\bullet}([C_2])=0$, then $u = v = 0$ and $\lambda_{1i}=0$ for all $1 \leq i \leq h_1$. Then $[C_2]= \gl_{1,h_1+1}[\q_{1,h_1+1}]=\gl_{11}[\p_{10}] + \sum_{i=1}^{h_1+1} \gl_{1i} [\q_{1i}] = [\omega_{R_2(Z_1)}]$.
\end{case}

\begin{case}
If $\rho^{\bullet\bullet}([C_1])=[\omega_{R_2(Z_0)}]= \lambda_{01} [\q_{01}]$, then $r=\lambda_{01}$ and $s=0$, so $[C_1]=\lambda_{01} [\q_{01}]=[\omega_{R_2(Z_0)}]$.

If  $\rho^{\bullet\bullet}([C_2])=[\omega_{R_2(Z_0)}]$, then $u=\lambda_{01}$, $v=0$, and $\lambda_{1i}=0$ for all $1 \leq i \leq h_1$, so that $[C_2]=\gl_{01}[\q_{01}]+\gl_{11}[\p_{10}] + \sum_{i=1}^{h_1+1} \gl_{1i} [\q_{1i}] = [\omega_{R}]$.
\end{case}

\begin{case}
If $\rho^{\bullet\bullet}([C_1])=[\omega_{R_2(Z_1^{\bullet\bullet})}]=\gl_{11}[\p_{10}] + \sum_{i=1}^{h_1} \gl_{1i} [\q_{1i}]$, then $r=0$, $s = \lambda_{11}$, and $\lambda_{1i}=0$ for all $1 \leq i \leq h_1$. So $[C_1]=0$.

If $\rho^{\bullet\bullet}([C_2])=[\omega_{R_2(Z_1^{\bullet\bullet})}]$, then $u=0$ and $v = \lambda_{11}$, so that $[C_2]=\gl_{11}[\p_{10}] + \sum_{i=1}^{h_1+1} \gl_{1i} [\q_{1i}]=[\omega_{R_2(Z_1)}]$.
\end{case}

\begin{case}
If $\rho^{\bullet\bullet}([C_1])=[\omega_{R^{\bullet\bullet}}]=\gl_{01}[\q_{01}]+\gl_{11}[\p_{10}] + \sum_{i=1}^{h_1} \gl_{1i} [\q_{1i}]$, then $r=\lambda_{01}$, $s = \lambda_{11}$, and $\lambda_{1i}=0$ for all $1 \leq i \leq h_1$. So $[C_1]=[\omega_{R_2(Z_0)}]$.

If $\rho^{\bullet\bullet}([C_2])=[\omega_{R^{\bullet\bullet}}]$, then $r=\lambda_{01}$ and $s = \lambda_{11}$, so that $[C_2]= [\omega_{R}]$.
\end{case}
\end{step}

We have established that $\mathfrak S_0(R)\subseteq\{[R],[\omega_{R_2(Z_0)}],
[\omega_{R_2(Z_1)}],[\omega_R]\}$,
and hence $|\mathfrak S_0(R)| \leq |\mathfrak S_0(R_2(Z_0))| \cdot |\mathfrak S_0(R_2(Z_1))|$.
The reverse inequality is given by Corollary~\ref{lowerbound}.
\end{proof}

\begin{lem} \label{Mtype1L}
 Let $Y = Z_0 \# Z_1$ be a 2-connected ladder, where $Z_1$ is a matrix of indeterminates and $Z_1$ is any 2-connected ladder 
 with no coincidental inside corners. Let $R=R_2(Y)$. Then $\mathfrak S_0(R)=\{[R],[\omega_{R_2(Z_0)}],[\omega_{R_2(Z_1)}],[\omega_R]\}$.
In particular, $|\mathfrak S_0(R)| = |\mathfrak S_0(R_2(Z_0))| \cdot |\mathfrak S_0(R_2(Z_1))|$.
\end{lem}

\begin{proof}[Proof-outline]
  The proof mimics that above, inducting on $h_1$.  The base case of $h_1 = 0$ is given by Lemma~\ref{onesidedtype1}, and the possible semidualizing modules are determined by the Two-Sided Ladder Theorem \cite{SWSeSpP1}.  Otherwise, the argument proceeds in a similar manner.
\end{proof}

\begin{lem} \label{lem:oo}
   Let $Y = Z_0 \# Z_1$ be a 2-connected ladder, where $Z_0,Z_1$ are one-sided ladders with $h_0=k_1=0$ or $k_0=h_1=0$. Let $R=R_2(Y)$. Then $\mathfrak S_0(R)=\{[R],[\omega_{R_2(Z_0)}],[\omega_{R_2(Z_1)}],[\omega_R]\}$.
In particular, $|\mathfrak S_0(R)| = |\mathfrak S_0(R_2(Z_0))| \cdot |\mathfrak S_0(R_2(Z_1))|$.
\end{lem}

\begin{proof}[Proof-outline]
  The proof is again similar.  It may be assumed that $h_0=k_1=0$, for if $Y = Z_0 \# Z_1$ with $k_0=h_1=0$, then $Y$ can be antitransposed.  Inducting on $h_1$, with the base case of $h_1 = 0$ given by Lemma~\ref{onesidedtype1}, proceed in the manner above.  
 \end{proof}

\begin{lem} \label{lem:ol}
  Let $Y = Z_0 \# Z_1$ be a 2-connected ladder, where $Z_0$ is a one-sided ladder and $Z_1$ is any 2-connected ladder with no coincidental inside corners. Let $R=R_2(Y)$. Then $\mathfrak S_0(R)=\{[R],[\omega_{R_2(Z_0)}],$ $[\omega_{R_2(Z_1)}],[\omega_R]\}$.
In particular, $|\mathfrak S_0(R)| = |\mathfrak S_0(R_2(Z_0))| \cdot |\mathfrak S_0(R_2(Z_1))|$.
\end{lem}

\begin{proof}
  By antitransposing if necessary, we may assume that $k_0=0$ for the ladder $Z_0$. The class group of $R$ has $\{[\q_{01}], \dots, [\q_{0,h_0+1}], [\q_{11}], \dots,$ $[\q_{1,h_1+1}],[\p_{10}], \dots, [\p_{1k_1}]\}$ as a basis; i.e., $\Cl(R) \cong \bbz^{h_0+h_1+k_1+3}$. The canonical class of $R$ is
\[
  [\omega_R] = \sum_{i=1}^{h_0+1}\gl_{0i}[\q_{0i}] + \sum_{i=1}^{h_1+1} \gl_{1i} [\frak q_{1i}] + \gl_{11}[\p_{10}] + \sum_{j=1}^{k_1}\delta_{1j} [\frakp_{1j}].
\]

We will prove the Lemma by double induction on $(h_0,h_1) \in \bbn^2$. First, we induct on $h_0$. The case $h_0 = 0$ is given by Lemma~\ref{Mtype1L}. So we may assume that $h_0>0$ and that this Lemma holds by induction for $(h_0-1,g)$ for any $g \in \bbn$. Next, we induct on $h_1$. The case $h_1=0$ is given by Lemma~\ref{lem:oo}. Thus, let $h_1 > 0$ and assume by induction that this Lemma holds for $(h_0,g)$ for all $g<h_1$. As before, $C_1,C_2,\dots$ will denote possible semidualizing modules of $R$.

\setcounter{step}{0}
\begin{step}
 Obtain the ladder $Y^{\bullet}$ by deleting rows $1,\dots,a_{01}-1$ and columns $b_{01}+1,\dots, b_{00}$ of $Y$. Then $Y^{\bullet}=Z_0^{\bullet} \# Z_1$, where $Z_0^{\bullet}$ is a one-sided ladder with one fewer inside corner than $Z_0$.  Invert $x_{1, b_{01}}$ in $R$.  Under the natural map $\rho^{\bullet} \colon \Cl(R) \to \Cl(R^{\bullet})$, where $R^{\bullet} = R_2(Y^{\bullet})$, we have $[\q_{01}] \mapsto 0$, and all other basis elements of $\Cl(R)$ are mapped to their same representations in $\Cl(R^{\bullet})$, so $\ker(\rho^{\bullet})=\bbz[\q_{01}]$. Since $Z_0^{\bullet}$ has $h_0-1$ lower inside corners, the induction hypothesis gives $\s_0(R^{\bullet})=\{0, [\omega_{R_2(Z_0^{\bullet})}], [\omega_{R_2(Z_1)}], [\omega_{R^{\bullet}}] \}$. Thus, $R$ has four possible classes of semidualizing modules $[C_1],[C_2],[C_3],[C_4]$, where
\begin{align*}
  [C_1] -  r[\q_{01}] &= 0,\\
  [C_2] - s[\q_{01}] &= \sum_{i=2}^{h_0+1} \gl_{0i} [\q_{0i}] = [\omega_{R_2(Z_0)}]-\gl_{01}[\q_{01}],\\
  [C_3] - u[\q_{01}] &= \gl_{11}[\p_{10}] + \sum_{i=1}^{h_1+1} \gl_{1i} [\q_{1i}] + \sum_{j=1}^{k_1} \gd_{1j}[\p_{1j}]
                                   = [\omega_{R_2(Z_1)}] \text{ and}\\
  [C_4] - v[\q_{01}] &= \sum_{i=2}^{h_0+1} \gl_{0i} [\q_{0i}] + \gl_{11}[\p_{10}] + \sum_{i=1}^{h_1+1} \gl_{1i} [\q_{1i}] + \sum_{j=1}^{k_1} \gd_{1j}[\p_{1j}]\\
                                 &= [\omega_R]-\gl_{01}[\q_{01}],
\end{align*}
such that $r,s,u,v \in \bbz$, and the classes $0, [\omega_{R_2(Z_0^{\bullet})}], [\omega_{R_2(Z_1)}], [\omega_{R^{\bullet}}]$ have the same representations in $\Cl(R^{\bullet})$ as on the right hand side.
\end{step}

\begin{step}
  Let $\kappa_1 = \min\{i \mid c_{1i} \geq a_{1,h_1}\}$, as in \cite[Notation 3.1]{SWSeSpP1}. We obtain $Y^{\bullet\bullet}$ by deleting rows $a_{1,h_1}+1,\dots, a_{1,h_1+1}$ and columns $1,\dots,b_{1,h_1}-1$ of $Y$. Then $Y^{\bullet\bullet} = Z_0 \# Z_1^{\bullet\bullet}$, where $Z_1^{\bullet\bullet}$ is a ladder with $h_1-1$ lower inside corners and $\kappa_1-1$ upper inside corners. Invert $x_{a_{1,h_1}, 1}$ in $R$.  Under the natural map $\rho^{\bullet\bullet} \colon \Cl(R) \to \Cl(R^{\bullet\bullet})$, we have $[\q_{1,h_1+1}],[\p_{1,\kappa_1}],\dots,[\p_{1,k}] \mapsto 0$, and all other basis elements of $\Cl(R)$ are mapped to their same representations in $\Cl(R^{\bullet\bullet})$.  Since $Z_1^{\bullet\bullet}$ has $h_1-1$ lower inside corners, the induction hypothesis gives $\s_0 (R^{\bullet\bullet})=\{0, [\omega_{R_2(Z_0)}], [\omega_{R_2(Z_1^{\bullet\bullet})}], [\omega_{R^{\bullet\bullet}}] \}$, where
\begin{align*}
  [\omega_{R_2(Z_0)}] &=  \sum_{i=1}^{h_0+1} \gl_{0i} [\q_{0i}],\\
  [\omega_{R_2(Z_1^{\bullet\bullet})}] &= \gl_{11}[\p_{10}] + \sum_{i=1}^{h_1} \gl_{1i} [\q_{1i}] + \sum_{j=1}^{\kappa_1-1} \gd_{1j}[\p_{1j}],
\end{align*}
and $[\omega_{R^{\bullet\bullet}}] = [\omega_{R_2(Z_0)}] + [\omega_{R_2(Z_1^{\bullet\bullet})}]$.
We consider several cases.

\setcounter{case}{0}
\begin{case}        
If $\rho^{\bullet\bullet}([C_1])=0$, then $r=0$, so $[C_1]=0$.

If $\rho^{\bullet\bullet}([C_2])=0$, then $s = 0$ and $\gl_{0i}=0$ for all $2 \leq i \leq h_0+1$, so $[C_2]=0$.

If $\rho^{\bullet\bullet}([C_3])=0$, then $u = 0$, $\gl_{1i}=0$ for all $1 \leq i \leq h_1$ and $\gd_{1j}=0$ for all $1 \leq j \leq \kappa_1-1$. Then $[C_3]=\gl_{11}[\p_{10}] + \sum_{i=1}^{h_1+1} \gl_{1i} [\q_{1i}] + \sum_{j=1}^{k_1} \gd_{1j}[\p_{1j}]=[\omega_{R_2(Z_1)}]$.

If $\rho^{\bullet\bullet}([C_4])=0$, then $v = 0$, $\gl_{0i}=0$ for all $2 \leq i \leq h_0+1$, $\gl_{1i}=0$ for all $1 \leq i \leq h_1$ and $\gd_{1j}=0$ for all $1 \leq j \leq \kappa_1-1$, so $[C_4]=[\omega_{R_2(Z_1)}]$.
\end{case}

\begin{case}
If $\rho^{\bullet\bullet}([C_1])=[\omega_{R_2(Z_0)}]$, then $r = \gl_{01}$ and $\gl_{0i} = 0$ for all $2 \leq i \leq h_0+1$. Then $[C_1] = \sum_{i=1}^{h_0+1} \gl_{0i} [\q_{0i}] = [\omega_{R_2(Z_0)}]$.

If $\rho^{\bullet\bullet}([C_2])=[\omega_{R_2(Z_0)}]$, then $s = \gl_{01}$, so $[C_2] = [\omega_{R_2(Z_0)}]$.

If $\rho^{\bullet\bullet}([C_3])=[\omega_{R_2(Z_0)}]$, then $u = \gl_{01}$, $\gl_{0i}=0$ for all $2 \leq i \leq h_0+1$, $\gl_{1i}=0$ for all $1 \leq i \leq h_1$ and $\gd_{1j}=0$ for all $1 \leq j \leq \kappa_1-1$, so $[C_3]=[\omega_R]$.

If $\rho^{\bullet\bullet}([C_4])=[\omega_{R_2(Z_0)}]$, then $v = \gl_{01}$, $\gl_{1i}=0$ for all $1 \leq i \leq h_1$ and $\gd_{1j}=0$ for all $1 \leq j \leq \kappa_1-1$, so $[C_4]=[\omega_R]$.
\end{case}

\begin{case}
If $\rho^{\bullet\bullet}([C_1])=[\omega_{R_2(Z_1^{\bullet\bullet})}]$, then $r=0$, so $[C_1]=0$.

If $\rho^{\bullet\bullet}([C_2])=[\omega_{R_2(Z_1^{\bullet\bullet})}]$, then $s=0$ and $\gl_{0i}=0$ for all $2 \leq i \leq h_0+1$, so $[C_2]=0$.

If $\rho^{\bullet\bullet}([C_3])=[\omega_{R_2(Z_1^{\bullet\bullet})}]$, then $u = 0$, so $[C_3]=[\omega_{R_2(Z_1)}]$. 

If $\rho^{\bullet\bullet}([C_4])=[\omega_{R_2(Z_1^{\bullet\bullet})}]$, then $v = 0$ and $\gl_{0i}=0$ for all $2 \leq i \leq h_0+1$, so $[C_4]=[\omega_{R_2(Z_1)}]$.
\end{case}

\begin{case}
If $\rho^{\bullet\bullet}([C_1])=[\omega_{R^{\bullet\bullet}}]$, then $r = \gl_{01}$ and $\gl_{0i}=0$ for all $2 \leq i \leq h_0+1$, so $[C_1] = [\omega_{R_2(Z_0)}]$.

If $\rho^{\bullet\bullet}([C_2])=[\omega_{R^{\bullet\bullet}}]$, then $s = \gl_{01}$, so $[C_2] = [\omega_{R_2(Z_0)}]$.

If $\rho^{\bullet\bullet}([C_3])=[\omega_{R^{\bullet\bullet}}]$, then $u = \gl_{01}$ and $\gl_{0i}=0$ for all $2 \leq i \leq h_0+1$, so $[C_3]=[\omega_R]$.

If $\rho^{\bullet\bullet}([C_4])=[\omega_{R^{\bullet\bullet}}]$, then $v = \gl_{01}$, so $[C_4]=[\omega_R]$.
\end{case}
\end{step}

We have now shown that $\mathfrak S_0(R)\subseteq\{[R],
[\omega_{R_2(Z_0)}],[\omega_{R_2(Z_1)}],[\omega_R]\}$,
and the reverse inclusion is given by Corollary~\ref{lowerbound}.
\end{proof}

\begin{prop} \label{Ltype1L}
  Let $Y = Z_0 \# Z_1$ be a 2-connected ladder, where $Z_0,Z_1$ are 2-connected ladders with no coincidental inside corners. Let $R=R_2(Y)$. Then $\mathfrak S_0(R)=\{[R],[\omega_{R_2(Z_0)}],$ $[\omega_{R_2(Z_1)}],[\omega_R]\}$.
In particular, $|\mathfrak S_0(R)| = |\mathfrak S_0(R_2(Z_0))| \cdot |\mathfrak S_0(R_2(Z_1))|$.
\end{prop}

\begin{proof}
  The proof mimics that of Lemma~\ref{lem:ol}, inducting on $(h_0,h_1) \in \bbn^2$. First, induct on $h_0$, where the case $h_0=0$ is given by Lemma~\ref{lem:ol}; then assume that $h_0>0$ and that the Proposition holds by induction for $(h_0-1,g)$ for any $g \in \bbn$. Next, induct on $h_1$, where the base case $h_1=0$ is given by Lemma~\ref{lem:ol}. With $h_1>0$, assume by induction that the Proposition holds for $(h_0,g)$ for all $g<h_1$.   One difference from the previous proof is in step 1.  Here we let $\kappa_2=\max\{j \mid d_{0j} \geq b_{01}\}$ as in \cite[Notation 3.1]{SWSeSpP1}, and obtain the ladder $Y^{\bullet}$ by deleting rows $1,\dots,a_{01}-1$ and columns $b_{01}+1,\dots, b_{00}$ of $Y$.  The remainder of the proof is now similar to that of Lemma~\ref{lem:ol}.

\end{proof}

%~~~~~~~~~~~~~~~~~~~~~~~~~~~~~~~~~~~~~~~~~~~~~~~~~~~~~~~~~~~~~~~~~~~~~~~
\subsection{The general argument}
%~~~~~~~~~~~~~~~~~~~~~~~~~~~~~~~~~~~~~~~~~~~~~~~~~~~~~~~~~~~~~~~~~~~~~~~

\begin{thm}  \label{laddertype1s} 
  Let $Y= Z_0 \# \cdots \# Z_w$ be a 2-connected ladder, where each $Z_u$ is a 2-connected ladder with no coincidental inside corners. Let $R=R_2(Y)$. Then $\mathfrak S_0(R)=\{\,\sum_{u=0}^w \te_u [\omega_{R_2(Z_u)}] \mid \te_u = 0 \text{ or } 1\}$. In particular, $|\mathfrak S_0(R)| = |\mathfrak S_0(R_2(Z_0))| \cdot \dots \cdot |\mathfrak S_0(R_2(Z_w))|=2^{\ve_0 + \dots + \ve_w}$, where $\ve_u=0$ if $R_2(Z_u)$ is Gorenstein and $\ve_u=1$ otherwise.
\end{thm}

\begin{proof}
  An outline of the proof of the Theorem is as follows. First, we induct on the number of coincidental inside corners $w \in \bbn$. The case $w=0$ is given by the Two-Sided Ladder Theorem \cite{SWSeSpP1}. The case $w=1$ is given by Proposition~\ref{Ltype1L}. Thus, we may assume that $w>1$.

  Next, for each $w>1$, we need to consider in turn the six different combinations of $Z_0$ and $Z_w$ as in Proposition~\ref{onematrixtype1} to Proposition~\ref{Ltype1L}, viz.\ $Z_0$ and $Z_w$ are matrices of indeterminates, $Z_0$ is a matrix of indeterminates and $Z_w$ is a 2-connected one-sided ladder, and so on, until $Z_0$ and $Z_w$ are arbitrary 2-connected ladders. The proofs for the six different combinations are similar, as outlined above, hence we will show here only the proofs of two combinations.

  Let us consider the case when $Z_0,Z_w$ are matrices of indeterminates, for example.

\setcounter{step}{0}
\begin{step} \label{step:getc}
   Obtain the ladder $Y^{\bullet}$ by deleting rows $1,\dots,a_{01}-1$ and columns $b_{01}+1,\dots, b_{00}$ of $Y$; that is, $Y^{\bullet} = Z_1\# \cdots \# Z_w$. Invert $x_{1, b_{01}}$ in $R$. Under the natural map $\rho^{\bullet} \colon \Cl(R) \to \Cl(R^{\bullet})$, where $R^{\bullet} = R_2(Y^{\bullet})$, we have $\ker(\rho^{\bullet}) = \bbz[\q_{01}] \oplus \bbz[\p_{10}]$. Let $[K] = r[\q_{01}]+s[\p_{10}] \in \Cl(R)$, where $r,s \in \bbz$. Induction on $w$ gives $\s_0(R^{\bullet})=\{\,\sum_{u=1}^w \te_u [\omega_{R_2(Z_u)}] \mid \te_u = 0 \text{ or } 1\}$. Hence the possible classes of semidualizing modules of $R$ have the form $[C]= [K]+\vf_1([\omega_{R_2(Z_1)}]-\gl_{11}[\p_{10}])+\sum_{u=2}^w \vf_u [\omega_{R_2(Z_u)}]$, where $\vf_u=0$ or 1 for all $1 \leq u \leq w$ (see, for example, Lemma~\ref{onesidedtype1}).
\end{step}

\begin{step}
  Next, obtain $Y^{\bullet\bullet}$ by deleting rows $a_{w0}+1,\dots, m$ and columns $1,\dots,b_{w0}-1$ of $Y$. Then $Y^{\bullet\bullet}=Z_0\# \cdots \# Z_{w-1}$.  Invert $x_{a_{w0}, 1}$  in $R$.
Under the natural map $\rho^{\bullet\bullet} \colon \Cl(R) \to \Cl(R^{\bullet\bullet})$, we have $[\omega_{R_2(Z_w)}] \mapsto 0$. Induction on $w$ gives $\s_0(R^{\bullet\bullet})=\{\,\sum_{u=0}^{w-1} \te_u [\omega_{R_2(Z_u)}] \mid \te_u = 0 \text{ or } 1\}$.
\end{step}

\begin{step}
   Let $[C]$ be as in Step~\ref{step:getc}. Assume that $\rho^{\bullet\bullet}([C]) \in \s_0(R^{\bullet\bullet})$. We show that we get candidates $\{ [C] \in \Cl(R) \mid [C] \text{ is a possible semidualizing module of }R\}=\{\,\sum_{u=0}^w \te_u [\omega_{R_2(Z_u)}] \mid \te_u = 0 \text{ or } 1\}$.

$(\subseteq)$: We solve the equation
\begin{align*}
  \rho^{\bullet\bullet}([K]+\vf_1([\omega_{R_2(Z_1)}]-\gl_{11}[\p_{10}])+\sum_{u=2}^w \vf_u [\omega_{R_2(Z_u)}])
  &= \sum_{u=0}^{w-1} \psi_u [\omega_{R_2(Z_u)}], \text{ or}\\
   r[\q_{01}]+s[\p_{10}]+\vf_1([\omega_{R_2(Z_1)}]-\gl_{11}[\p_{10}])+\sum_{u=2}^{w-1} \vf_u [\omega_{R_2(L_u)}]
  &= \sum_{u=0}^{w-1} \psi_u [\omega_{R_2(Z_u)}],
\end{align*}
where $\vf_u,\psi_u = 0$ or 1. We need to find the coefficient of $[\omega_{R_2(Z_u)}]$ in $[C]$ as in Lemma~\ref{onesidedtype1}, for example. We show how to find the coefficient of $[\omega_{R_2(Z_1)}]$ in $[C]$ in the case $\vf_1 \neq \psi_1$.

If $\vf_1=0$ and $\psi_1=1$, then $s=\gl_{11}$, $\gl_{1i}=0$ for all $1 \leq i \leq h_1+1$ and $\gd_{1j}=0$ for all $1 \leq j \leq k_1$. So $s=0$, and the coefficient of $[\omega_{R_2(Z_1)}]$ in $[C]$ is 0.

If $\vf_1=1$ and $\psi_1=0$, then $s=0$, $\gl_{1i}=0$ for all $1 \leq i \leq h_1+1$ and $\gd_{1j}=0$ for all $1 \leq j \leq k_1$. So the coefficient of $[\omega_{R_2(Z_1)}]$ in $[C]$ is 0.

The coefficient of $[\omega_{R_2(Z_u)}]$ in $[C]$ is easier to find in all other cases. We get $[C]=\psi_0[\omega_{R_2(Z_0)}] + \sum_{u=1}^{w-1}\min(\vf_u,\psi_u)[\omega_{R_2(Z_u)}] + \vf_w[\omega_{R_2(Z_w)}]$, and certainly $\min(\vf_u,\psi_u) = 0$ or 1.

$(\supseteq)$: If $[D]=\sum_{u=0}^w \te_u [\omega_{R_2(Z_u)}] \in \Cl(R)$, where $\te_u = 0$ or 1, then
$\rho^{\bullet\bullet}([D])= \sum_{u=0}^{w-1} \te_u [\omega_{R_2(Z_u)}] \in \s_0(R^{\bullet\bullet})$. Hence $[D]$ is a possible semidualizing module of $R$.

Corollary~\ref{lowerbound} now shows that all modules in $\{\,\sum_{u=0}^w \te_u [\omega_{R_2(Z_u)}] \mid \te_u = 0 \text{ or } 1\}$ are in fact semidualizing modules of $R$.
\end{step}

  Next, we consider the case when $Z_0,Z_w$ are arbitrary 2-connected ladders, %Assume that we have proved the Theorem
   %for all five other possible combinations of $Z_0$ and $Z_w$, and that we are inducting on $(h_0,h_w) \in \bbn^2$, as in   Proposition~\ref{Ltype1L}, with $h_0,h_w > 0$. The induction is as follows.
assuming that the Theorem holds for all five other possible combinations of $Z_0$ and $Z_w$.  We induct on $(h_0,h_w) \in \bbn^2$, as in Proposition~\ref{Ltype1L}, with $h_0,h_w > 0$.

\setcounter{step}{0}
\begin{step}
  Let $\kappa_2=\max\{j \mid d_{0j} \geq b_{01}\}$.
  Obtain the ladder $Y^{\bullet}$ by deleting rows $1,\dots,a_{01}-1$ and columns $b_{01}+1,\dots, b_{00}$ of $Y$. Then $Y^{\bullet}=Z_0^{\bullet} \# Z_1 \# \cdots \# Z_w$, where $Z_0^{\bullet}$ has one fewer lower inside corner than $Z_0$.  Invert $x_{1, b_{01}}$ in $R$.  Under the natural map $\rho^{\bullet} \colon \Cl(R) \to \Cl(R^{\bullet})$, where $R^{\bullet} = R_2(Y^{\bullet})$, we have $\ker(\rho^{\bullet})=\bbz[\q_{01}]\oplus \bbz[\p_{01}] \oplus \dots \oplus \bbz[\p_{0\kappa_2}]$. Let $K=r[\q_{01}]+ \sum_{j=1}^{\kappa_2}s_j[\p_{0j}]$, where $r,s_j \in \bbz$. Induction on $h_0$ gives $\s_0(R^{\bullet})=\{\,\te_0 [\omega_{R_2(Z_0^{\bullet})}] + \sum_{u=1}^w \te_u [\omega_{R_2(Z_u)}] \mid \te_u = 0 \text{ or } 1\}$. Thus, the possible classes of semidualizing modules of $R$ have the form $[C]= [K]+\vf_0([\omega_{R_2(Z_0)}]-\gl_{01}[\q_{01}]-\sum_{j=1}^{\kappa_2} \gd_{0j}[\p_{0j}])+\sum_{u=1}^w \vf_u [\omega_{R_2(Z_u)}]$, where $\vf_u=0$ or 1 for all $0 \leq u \leq w$, as in Proposition~\ref{Ltype1L}.
\end{step}

\begin{step}
  Let $\kappa_1 = \min\{i \mid c_{wi} \geq a_{w,h_w}\}$. We obtain $Y^{\bullet\bullet}$ by deleting rows $a_{w,h_w}+1,\dots, a_{w,h_w+1}$ and columns $1,\dots,b_{w,h_w}-1$ of $Y$. Then $Y^{\bullet\bullet} = Z_0 \# \cdots \# Z_{w-1} \# Z_w^{\bullet\bullet}$, where $Z_w^{\bullet\bullet}$ is a ladder with $h_w-1$ lower inside corners. Invert $x_{a_{w,h_w}, 1}$ in $R$.  Under the natural map $\rho^{\bullet\bullet} \colon \Cl(R) \to \Cl(R^{\bullet\bullet})$, we have $\ker(\rho^{\bullet\bullet})=\bbz [\q_{w,h_w+1}] \oplus \bbz [\p_{w,\kappa_1}] \oplus \dots \oplus \bbz [\p_{w,k_w}]$. Induction on $h_w$ gives $\s_0 (R^{\bullet\bullet})=\{\,\te_w [\omega_{R_2(Z_w^{\bullet\bullet})}] + \sum_{u=0}^{w-1} \te_u [\omega_{R_2(Z_u)}] \mid \te_u = 0 \text{ or } 1\}$.
\end{step}

\begin{step}
  Assume that $\rho^{\bullet\bullet}([C]) \in \s_0(R^{\bullet\bullet})$. We show that we get candidates $\{ [C] \in \Cl(R) \mid [C] \text{ is a possible semidualizing module of }R\}=\{\,\sum_{u=0}^w \te_u [\omega_{R_2(Z_u)}] \mid \te_u = 0 \text{ or } 1\}$.

$(\subseteq)$: We solve the equation $\rho^{\bullet\bullet}([C])
  = \psi_w [\omega_{R_2(Z_w^{\bullet\bullet})}] + \sum_{u=0}^{w-1} \psi_u [\omega_{R_2(Z_u)}]$, where $\rho^{\bullet\bullet}([C])$ equals
\[
  [K] + \vf_0 \left( [\omega_{R_2(Z_0)}]-\gl_{01}[\q_{01}]-\sum_{j=1}^{\kappa_2} \gd_{0j}[\p_{0j}] \right) + \sum_{u=1}^{w-1} \vf_u [\omega_{R_2(Z_u)}] + \vf_w [\omega_{R_2(Z_w^{\bullet\bullet})}],
\]
and $\vf_u,\psi_u=0$ or 1. Then  $[C]=\psi_0[\omega_{R_2(Z_0)}] + \sum_{u=1}^{w-1}\min(\vf_u,\psi_u)[\omega_{R_2(Z_u)}] + \vf_w[\omega_{R_2(Z_w)}]$ (see Lemma~\ref{lem:ol}), and certainly $\min(\vf_u,\psi_u) = 0$ or 1.

$(\supseteq)$: If $[D]=\sum_{u=0}^w \te_u [\omega_{R_2(Z_u)}] \in \Cl(R)$, where $\te_u = 0$ or 1, then
$\rho^{\bullet\bullet}([D])= \te_w [\omega_{R_2(Z_w^{\bullet\bullet})}] + \sum_{u=0}^{w-1} \te_u [\omega_{R_2(Z_u)}] \in \s_0(R^{\bullet\bullet})$. Hence $[D]$ is a possible semidualizing module of $R$.
\end{step}

Again Corollary~\ref{lowerbound} completes the induction on $w$.
\end{proof}

\begin{cor} \label{2N} For any $N \in \mathbb N$, there exist ladders $Y$ such that $|\mathfrak S_0(R_2(Y))| = 2^N$.  In fact, infinitely many such ladders exist.
\end{cor}

\begin{proof} Let $N \in \mathbb N$, and for $i = 0, \dots, N-1$, let $m_i, n_i > 1$ be pairwise distinct integers.  Let $Z_i$ be the matrix of variables of size $m_i \times n_i$.  Then each (ladder) determinantal ring $R_2(Z_i)$ is not Gorenstein (see Fact \ref{Gorprop}).  Setting $Y = Z_0 \# \cdots \# Z_{N-1}$, it follows from Theorem \ref{laddertype1s} that $|\frak S_0(R_2(Y))| = 2^N$.  The same result holds for more general ladders: let $Z_i$ be a ladder of size $m_i \times n_i$ where $m_i, n_i$ are (not necessarily distinct) integers greater than 1. If $m_i \neq n_i$, or all inside corners $(r,s)$ of $Z_i$ satisfy $r+s = m_i+1$, then $R_2(Z_i)$ is not Gorenstein, per Fact \ref{Gorprop}.
\end{proof}

\noindent {\bf Example~\ref{type1} (concluded).}  For the ladder $L_3 = Z_0\# Z_1$, of two $3 \times 2$ matrices, in the Introduction, 
 $|\frak S_0(R_2(L_3))| = 4$.  Set $R = R_2(L_3)$.  Then $\frak S_0(R) = \{[R], [(x_{12},x_{13})], [(x_{31},x_{32})],[\omega_R]\}$.

\begin{ex} Finally, for the ladder $Y = L_1 \# L_2 \# L_3$, where the ladders $L_i$ are those from the Introduction, $|\frak S_0(R_2(Y))| = 8$ by Theorem \ref{laddertype1s}.  More specifically, only $R_2(L_2)$ is Gorenstein (see Fact \ref{Gorprop}) and $|\frak S_0(R_2(L_1))| = 2$ by the Two-Sided Ladder Theorem \cite{SWSeSpP1}.  The previous example finishes off the calculation.
  \end{ex}

\providecommand{\bysame}{\leavevmode\hbox to3em{\hrulefill}\thinspace}
\providecommand{\MR}{\relax\ifhmode\unskip\space\fi MR }
\providecommand{\MRhref}[2]{%
  \href{http://www.ams.org/mathscinet-getitem?mr=#1}{#2}
}
\providecommand{\href}[2]{#2}

\end{document}